\documentclass[10pt,oneside]{amsart}
\usepackage[a4paper,width=15cm,bottom=4cm]{geometry}
\usepackage[utf8x]{inputenc}
\usepackage[english]{babel}
\usepackage{graphicx}
\usepackage{lmodern}
\usepackage{hyperref}
\usepackage[short-journals,initials]{amsrefs}
\usepackage{xspace}
\usepackage{enumerate}
\usepackage{tikz}
\usepackage{microtype}

\newcommand{\Sinai}{Sina{\u{\i{}}}\xspace}

\newcommand{\Barany}{B\'ar\'any\xspace}
\newcommand{\Jarnik}{Jarn\'{\i}k\xspace}
\newcommand{\Zunic}{{\v{Z}}uni{\'c}\xspace}
\newcommand{\Erdos}{Erd{\H o}s\xspace}
\newcommand{\emphset}[1]{\mathbb{#1}}

\newcommand{\ZZ}{\emphset Z}
\newcommand{\Z}{\emphset{Z}}
\newcommand{\N}{\emphset{N}}
\newcommand{\PP}{\emphset{P}}
\newcommand{\EE}{\emphset{E}}
\newcommand{\RR}{\emphset{R}}
\newcommand{\CC}{\emphset{C}}

\newcommand{\bigO}{{O}}
\newcommand{\Ecal}{\mathcal{E}}

\DeclareMathOperator{\Li}{Li}

\newcommand{\XX}{\mathbb{X}}

\renewcommand{\L}{\XX}

\newcommand{\Lstar}{{\L}}
\newcommand{\ds}{\displaystyle}
\newcommand{\PPbl}{\PP_{\beta,\lambda}}

\newcommand{\EEbl}{\EE_{\beta,\lambda}}

\newcommand{\Gammabl}{\Gamma_{\beta,\lambda}}
\newcommand{\sigmabl}{\sigma_{\beta,\lambda}}
\newcommand{\Lbl}{L_{\beta,\lambda}}
\newcommand{\Zbl}{Z(\beta,\lambda)}
\newcommand{\phibl}{\phi_{\beta,\lambda}}

\def\e{{\bf e}}
\def\cc{{\bf c}}

\def\LL{\rm{Li}_2}
\def\LLL{\rm{Li}_3}

\renewcommand{\epsilon}{\varepsilon}

\newcommand{\indicator}[1]{\mathbf{1}_{\{#1\}}}

\newtheorem{theorem}{Theorem}
\newtheorem*{theorem*}{Theorem}
\newtheorem{lemma}{Lemma}
\newtheorem*{lemma*}{Lemma}

\theoremstyle{remark}
\newtheorem*{remark}{Remark}
\DeclareMathOperator{\Cov}{Cov}

\title{Convex lattice polygonal lines with a constrained number of vertices}

\author[J. Bureaux]{Julien Bureaux}
\author[N. Enriquez]{Nathana\"el Enriquez}
\address{MODAL'X, Université Paris Ouest Nanterre La Défense, 200 avenue de la République, 92001 Nanterre}
\email{julien.bureaux@math.cnrs.fr, nathanael.enriquez@u-paris10.fr}

\begin{document}

\begin{abstract}
A detailed combinatorial analysis of planar convex lattice polygonal lines
is presented.
This makes it possible to answer an open question of Vershik regarding the existence of a limit shape when the number of vertices is constrained.
\end{abstract}

\maketitle

\section{Introduction}

In 1979, Arnold~\cite{arnold_statistics_1980} considered the question of the number of
equivalence classes of convex lattice polygons having a given integer as area
(we say that two polygons having their vertices on $\Z^2$ are equivalent if one
is the image of the other by an automorphism of $\Z^2$).
Later, Vershik changed the constraint in this problem and raised the question of
the number, and typical shape, of convex lattice polygons included in a large box
$[-n,n]^2$. 
The stepping stone in this problem lies in the understanding of the number and shape of polygonal lines having integer vertices, starting from the origin and forming a sequence of increasing slopes.
 In 1994, three different solutions to this problem were found by
 B\'ar\'any~\cite{barany_limit_1995}, Vershik~\cite{vershik_limit_1994} and \Sinai{}~\cite{sinai_probabilistic_1994}. 
Namely, they proved that, when $n$ goes to infinity:
\begin{enumerate}[\quad\bfseries (a)]\itshape
    \item The number of convex polygonal lines with vertices in $(\ZZ\cap [0,n])^2$ joining $(0,0)$ to $(n,n)$ is equal to $\ds \exp(3(\zeta(3)/\zeta(2))^{1/3}\,n^{2/3}+o(n^{2/3}))$.
    \item The number of vertices constituting a typical line is equivalent to $(\zeta(3)^2\zeta(2))^{-1/3}\,n^{2/3}$.
    \item There is a limit shape for a typical convex polygonal line, which is an arc of a parabola.
\end{enumerate}

It turns out that these problems are related to an earlier family of works we shall discuss now. 
In 1926, \Jarnik{} found an asymptotic equivalent of the maximal number of integral points
that can be interpolated by a convex curve of Euclidean length $n$.
He obtained also an explicit number-theoretic constant times $n^{2/3}$.
This article was at the origin of many works of Diophantine analysis, and we
refer the reader to the papers of Schmidt~\cite{schmidt_integer_1985} and Bombieri and Pila~\cite{bombieri_number_1989} for more recent results,  discussions and open questions on this subject.
One may  slightly change Jarn\'{\i}k's framework, and consider the set of integral points which are interpolated by the graph on $[0,n]$ of  an increasing and strictly convex function satisfying $f(0)=0$ and $f(n)=n$.
In 1995, Acketa and \Zunic~\cite{acketa_maximal_1995} proved the following box analog of Jarn\'ik's result: the largest  number of vertices for an increasing convex polygonal line on $\Z_+^2$ joining $(0,0)$ to $(n,m)$ is asymptotically equivalent to $3\pi^{-2/3}({nm})^{1/3}$.
They derived the asymptotic value of the maximal number of vertices for a lattice polygon included in a square.

The nature of the results shows that these problems are related to both 
affine differential geometry and geometry of numbers. Indeed, the parabola
found as limit shape coincides with the convex curve inside the square having the
largest affine perimeter. Furthermore, the appearance of the values of the
Riemann zeta function underlines the arithmetic aspects of the problem. One could
show indeed, by using Valtr's formula~\cite{valtr_probability_1995}, that if the lattice $\Z^2$ was replaced by a Poisson Point Process
having intensity one (which can be thought as the most isotropic ``lattice'' one
can imagine), the constants $ (\zeta^2(3)\zeta(2))^{-1/3} \approx 0.749$ and 
$3(\zeta(3)/\zeta(2))^{1/3} \approx 2.702$ would be merely raised respectively to 1 and 3
asymptotically almost surely. The link with number theory was made even more clear by the authors who proved in \cite{bureaux_enriquez_2016} that Riemann's Hypothesis is actually equivalent to the fact that the remainder term $o(n^{2/3})$ in point \textbf{(a)} is $o(n^{1/6+\epsilon})$ for all $\epsilon>0$. 

As we said above, various strategies have been considered for Vershik's problem. \Barany{}~\cite{barany_limit_1995} and Vershik~\cite{vershik_limit_1994} use generating functions and an affine perimeter maximization problem.
Later, Vershik and Zeitouni~\cite{vershik_large_1999} made result \textbf{(c)} more precise and general by proving a large deviation principle whose rate function involves this affine perimeter.
\Sinai{}'s approach was very different. His proof is based on a statistical mechanical description of the problem. It was recently made fully rigorous and extended by Bogachev and Zarbaliev \cite{bogachev_universality_2011}.

\subsection{Main results}

Our aim in this paper is to improve the three results \textbf{(a),(b),(c)} described above.
In particular, we shall address the following natural extension of \textbf{(c)} which appears as an open question in Vershik's 1994 article:
\begin{quotation}
\itshape
``Theorem 3.1 shows how the number of vertices of a typical polygonal line grows. However, one can consider some other fixed growth, say, $\sqrt n$, and look for the limit shapes for uniform distributions connected with this growth [...]"
\end{quotation}
One of our results is that, not only there still exists a limit shape when the number of vertices is constrained, but also the parabolic limit shape is actually universal for all growth rates. 
The following theorem is a consequence of Theorem~\ref{thm:limit_shape_numerous} of section~\ref{sec:limit_shape} and Theorem~\ref{thm:limit_shape_few} of section~\ref{sec:few} which concern respectively limit shape results for lines with many and few vertices.

\begin{theorem*}
    The Hausdorff distance between a random convex polygonal line on $(\frac{1}{n}\ZZ\cap [0,1])^2$ joining $(0,0)$ to $(1,1)$ with at most $k$ vertices, and the arc of parabola
    \[
        \left\{(x,y) \in [0,1]^2 \mid \sqrt{\vphantom1 y}\,+ \sqrt{\vphantom y1-x} = 1\right\},
    \]
    converges in probability to $0$ when both $n$ and $k$ tend to infinity. 
\end{theorem*}

The proof of this theorem requires a detailed combinatorial analysis of convex polygonal lines with a constrained number of vertices. This is the purpose of Theorem~\ref{thm:detailed_comb} and Theorem~\ref{thm:asymp_very_few_vertices} which together complete results \textbf{(a)} and \textbf{(b)} in the following way:

\begin{theorem*}

Let $p(n;k)$ denote the number of convex polygonal lines in $\mathbb Z_+^2$ joining $(0,0)$ to $(n,n)$ and having  \(k\) vertices.

    \begin{itemize}
        \item 
            There exist two functions $\cc$ and $\e$ (which are explicitly computed in Theorem~\ref{thm:detailed_comb}) such that, for all  $\ell \in (0,+\infty)$, if $k$ is asymptotically equivalent to $\cc(\ell)\,n^{2/3}$, then
    \[
        \log p(n;k) \sim  \e(\ell)\, n^{2/3}.
    \]
        \item  If $k$ is asymptotically negligible compared to $n^{2/3}$, then
            \[
                p(n;k) = \left(\frac{n^2}{k^3}\right)^{k+o(k)}.
            \]
       \item If  $k$ is asymptotically negligible compared to  $n^{1/2}(\log n)^{-1/4}$, then
                \[
     p(n;k) \sim \frac{1}{k!}\binom{n-1}{k-1}^2.
    \]
    \end{itemize}
\end{theorem*}

Let us mention that the question of the number of vertices is reminiscent of other ones considered, for instance, by \Erdos and Lehner~\cite{erdos_distribution_1941}, Arratia
and Tavar\'e~
\cite{arratia_independent_1994}, or Vershik and Yakubovich~\cite{vershik_limit_2001} who were studying
combinatorial objects (integer partitions, permutations, polynomials over finite field, Young
tableaux, etc.) having a specified number of summands 
(according to the setting, we  call summands, cycles, irreducible divisors, etc.).

\subsection{Organization of the paper}
In section~\ref{sec:dca}, we present the detailed combinatorial analysis in the case of many vertices $k \gg \log |n|$.
Following \Sinai{}'s approach, the method, borrowed from classical
ideas of statistical physics, relies  on the introduction of a grand
canonical ensemble which endows the considered combinatorial object with a parametrized probability measure. Then, the strategy consists in calibrating the parameters of the probability in order to
fit with the constraints one has to deal with. Namely, in  our question, it turns
out that one can add one parameter in \Sinai{}'s probability distribution that makes it
possible to take into account, not only the location of the extreme point of the
polygonal line but also the number of vertices it contains.
In this model, we are able to establish a contour-integral representation of the logarithmic partition function in terms of Riemann's and Barnes' zeta functions. The residue analysis of this representation leads to precise estimates of this function as well as of its derivatives, which correspond to the moments of the random variables of interest such as the position of the terminal point and the number of vertices of the line. Using a local limit theorem, we finally obtain the
asymptotic behavior of the number of lines having $\cc(\ell)\, (n_1n_2)^{1/3}$ vertices in terms of the polylogarithm functions $\Li_1,\Li_2,\Li_3$. We also obtain an asymptotic formula for the number of lines having a number $k$ of vertices satisfying $\log |n| \ll k \ll |n|^{2/3}$.

In section~\ref{sec:limit_shape}, we derive results about the limit shape of lines having a fixed number of vertices $k \gg \log |n|$, answering the question of Vershik in a wide range.

In section~\ref{sec:few}, we extend the results about combinatorics and limit shape beyond $\log |n|$. The approach here is radically different and more elementary. It allows us to recover the results of sections~\ref{sec:dca} and~\ref{sec:limit_shape}, up to $k \ll |n|^{1/3}$. It relies on the comparison with a continuous setting which has been studied by \Barany~\cite{barany_sylvesters_1999} and \Barany, Rote, Steiger, Zhang \cite{barany_central_2000}.

In section~\ref{sec:jarnik}, we go back to Jarn\'{\i}k's original problem. In addition to Jarn\'{\i}k's result that we recover, we give the asymptotic number of lines, typical number of vertices, and  limit shape, which is an arc of a circle, in this different framework. 

In section~\ref{sec:onion}, we mix both types of conditions. The statistical physical method still applies and we obtain, for the convex lines joining $(0,0)$ to $(n,n)$ and having  a given total length, a continuous family of convex limit shapes that interpolates the diagonal of the square and the two sides of the square, going through the above arc of parabola and  arc of  circle.

\section{A one-to-one  correspondence}
\label{sec:correspondence}

We start this paper by reminding the correspondence between finite convex polygonal
lines issuing from $0$ whose vertices define increasing sequences in both
coordinates and finite distributions of multiplicities on the set of pairs of coprime positive integers.
This correspondence is a discrete analogue of the Gauss-Minkowski transformation in convex geometry.

More precisely,  let $\Pi$ denote the set of finite planar convex polygonal lines $\Gamma$ issuing
from 0 such that the vertices of $\Gamma$ are points of the lattice $\Z^2$ and the angle
between each side of $\Gamma$ and the horizontal axis is in the interval
$[0,\pi/2]$.  Now consider the set $\XX$ of all vectors $x = (x_1,x_2)$ whose coordinates are coprime positive integers including the pairs $(0,1)$ and $(1,0)$. \Jarnik{} observed that the space $\Pi$ admits a simple alternative description in terms of distributions of multiplicities on $\XX$.

\begin{lemma*}
\label{lem:corresp}
    The space \(\Pi\) is in one-to-one correspondence with the space
    \(\Omega\) of nonnegative integer-valued functions \(x \mapsto
    \omega(x)\) on \(\XX\) with finite support (that is \(\omega(x)\neq
    0\) for only finitely many \(x \in \XX\)).
\end{lemma*}

The inverse map \(\Omega \to \Pi\) corresponds to the following simple construction: for a given multiplicity distribution \(\omega \in \Omega\) and for all \(\theta\in [0,\infty]\), let us define
\begin{equation}
    X_i^\theta(\omega) := \sum_{\substack{(x_1,x_2) \in \XX\\ x_2 \leq \theta x_1}} \omega(x)\cdot x_i, \qquad i \in \{1,2\}.
\end{equation}
When \(\theta\) ranges over \([0,\infty]\), the function \(\theta \mapsto X^\theta(\omega) = (X^\theta_1(\omega),X^\theta_2(\omega))\) takes a finite number of values which are points of the lattice quadrant \(\ZZ_+^2\). These points are in convex position since we are adding vectors in increasing slope order. The convex polygonal curve \(\Gamma \in \Pi\) associated to \(\omega\) is simply the linear interpolation of these points starting from \( (0,0) \).

\section{A detailed combinatorial analysis}
\label{sec:dca}

For every \(n = (n_1,n_2) \in \ZZ_+^2\) and \(k\in \ZZ_+\), define \(\Pi(n; k)\) the subset of $\Pi$ consisting of polygonal lines \(\Gamma \in \Pi\) with endpoint \(n\)
and having  \(k\) vertices, and denote by \(p(n;k) := \left|\Pi(n;k)\right|\) its cardinality. Before we can state our first theorem, let us recall that the polylogarithm $\Li_s(z)$ is defined for all complex number $s$ with $\Re(s) > 0$ and $|z| < 1$ by
\[
    \Li_s(z) = \sum_{k=1}^\infty \frac{z^k}{k^s} = \frac{1}{\Gamma(s)}\int_0^\infty \frac{z t^{s-1}}{e^t - z} dt.
\]
The integral in the last term is a holomorphic function of $z$ in $\CC\setminus[1,+\infty)$. We will work with this analytic continuation of $\Li_s$ in the sequel.
We now define $\cc(\ell)$ and $\e(\ell)$ for all $\ell\in (0,+\infty)$ by
\[
    \cc(\ell) =\frac{\ell}{1-\ell}
    \times\frac{\LL(1-\ell)}{\zeta(2)^{1/3}(\zeta(3)- 
    \LLL(1-\ell))^{2/3}},
    \quad
    \e(\ell) =3\left(\frac{\zeta(3)- 
    \LLL(1-\ell)}{\zeta(2)}\right)^{1/3}- \log(\ell)\cc(\ell).
\]
The following statement indicates the asymptotic exponential  behavior of $p(n;k)$ in the case of many vertices, that is to say, when $k$ is not too small with respect to $|n|$.

\begin{theorem}
    \label{thm:detailed_comb}
    Suppose that $|n|$ and $k$ tend to $+\infty$ such that $n_1 \asymp n_2$ and $\log |n|$ is asymptotically negligible compared to $k$.
    
    \begin{itemize}
        \item 
            If there exists $\ell \in (0,+\infty)$ such that $k \sim \cc(\ell)(n_1n_2)^{1/3}$, then
    \[
        \log p(n;k) \sim  \e(\ell) (n_1n_2)^{1/3}.
    \]
        \item  If $k$ is asymptotically negligible compared to $(n_1n_2)^{1/3}$, then
            \[
                p(n;k) = \left(\frac{n_1n_2}{k^3}\right)^{k+o(k)}.
            \]
    \end{itemize}
\end{theorem}

\begin{figure}[h]
    \begin{center}
      \includegraphics{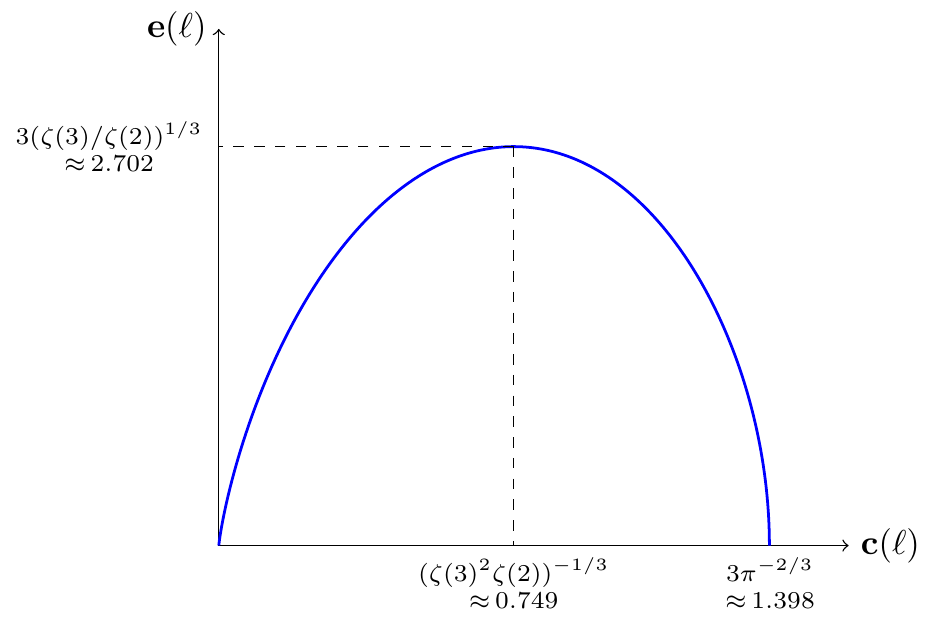}
      \caption[Distribution of the number of vertices]{Distribution of the number of vertices of a random convex polygonal line. The point of maximal $\e$-coordinate corresponds to typical lines. The point of maximal $\cc$-coordinate corresponds to lines with a maximal number of vertices. Note that the curve is not symmetric.}
    \end{center}
\end{figure}

\begin{remark}
    The function $\ell \mapsto \e(\ell)$ is maximal for $\ell=1$ and the corresponding coefficients are
    \begin{eqnarray*}
        \cc(1)=\dfrac{1}{(\zeta(2)\zeta(3)^2)^{1/3} },
         &
    \e(1)=3\left(\dfrac{\zeta(3)}{\zeta(2)}\right)^{1/3},
    \end{eqnarray*}
    which already recovers results \textbf{(a)} and \textbf{(b)}.
\end{remark}

\begin{remark}
	\label{rem:longest}
	As a byproduct of Theorem~\ref{thm:detailed_comb}, one can deduce the asymptotic behavior of the maximal number $M(n)$ of integral points that an increasing convex function satisfying $f(0)=0$ and $f(n)=n$ can interpolate. This question and its counterpart, concerning the maximal convex lattice polygons inscribed in a convex set was solved by Acketa and \Zunic~\cite{acketa_maximal_1995} who proved that $M_n \sim
3\pi^{-2/3}\,n^{2/3}$.

Starting from Theorem~\ref{thm:detailed_comb}, the proof goes as follows. We first notice that $\e(\lambda)$ tends to $0$ when $\lambda$ goes to
infinity. In the same time,
$ \cc(\lambda) \sim -\zeta(2)^{-1/3}\LL(1-\lambda)(-\LLL(1-\lambda))^{-2/3}$
which tends to ${3}\pi^{-2/3}$.
Since $\e(\lambda)$ remain strictly positive, we get $\liminf n^{-2/3}M(n) \geq
    3\pi^{-2/3}$.
    Now, let $\epsilon>0$ and suppose
    $ \limsup n^{-2/3} M(n) \geq 3\pi^{-2/3}(1+2\epsilon)$.
Then, for arbitrary large $n$, there is a polygonal line $\Gamma\in\Pi(n,n)$ having at least $3\pi^{-2/3}\,n^{2/3}(1+\epsilon)$ vertices. By choosing $k = 3\pi^{-2/3}\,n^{2/3}$ vertices among the vertices of this line, we get already a subset of $\Pi(n; k)$ whose cardinality is larger than $e^{cn^{2/3}}$ with $c > 0$. This enters in contradiction with the fact that $\lim_{\lambda\to\infty} \e(\lambda) = 0$.
\end{remark}

\subsection{Modification of \Sinai{}'s model and proof of Theorem~\ref{thm:detailed_comb}}

Recall from section~\ref{sec:correspondence} the set $\XX = \{(x_1,x_2) \in \Z_+^2 \mid \gcd(x_1,x_2) = 1\}$ of primitive vectors and the set $\Omega$ of functions $\omega \colon \XX \to \Z_+$ with finite support.
The restriction of \Jarnik's correspondence to the subspace $\Pi(n;k)$ induces a bijection with the subset \(\Omega(n;k)\) of \(\Omega\) consisting of multiplicity distributions \(\omega \in \Omega\) such that the ``observables''
\[
    X_1(\omega) := \sum_{x \in \XX} \omega(x)\cdot x_1,\quad
    X_2(\omega) := \sum_{x \in \XX} \omega(x)\cdot x_2,\quad 
    K(\omega) := \sum_{x \in \XX} \indicator{\omega(x) > 0}
\]
are respectively equal to \(n_1,n_2\) and \(k\). Notice that $X_1 = X_1^\infty$ and $X_2 = X_2^\infty$ with the previous notations. The random variables $X_1$ and $X_2$ correspond to the coordinates of the endpoint of the polygonal chain while $K$ counts its number of vertices.

For all $\lambda > 0$ and for every couple of parameters $\beta = (\beta_1,\beta_2) \in (0,+\infty)^2$, we endow \(\Omega\) with the probability measure defined for $\omega \in \Omega$ by
\begin{align*}
    \PPbl(\omega) & := \frac{1}{\Zbl}\exp\left[- \sum_{x \in \L} \omega(x)\, \beta \cdot x\right] \lambda^{K(\omega)}\\
                  & = \frac{1}{\Zbl} e^{-\beta_1 X_1(\omega)} e^{-\beta_2 X_2(\omega)} \lambda^{K(\omega)},
\end{align*}
where the \emph{partition function} $\Zbl$ is chosen as the normalization constant
\begin{equation}
	\label{eq:Zbl}
    \Zbl = \sum_{n \in \Z_+^2} \sum_{k\geq1} p(n ; k)\, e^{-\beta \cdot n} \lambda^k.
\end{equation}
Note that $\Zbl$ is finite for all values of the parameters $(\beta,\lambda) \in (0,+\infty)^3$. Indeed, if we denote by $p(n) = \sum_{k\geq 1} p(n;k)$ the total number of convex polygonal lines of $\Pi$ with end point $n = (n_1,n_2)$ and $M_n$ the maximal number of edges of such a line, the following bound holds:
\[
    \Zbl \leq \sum_{n \in \Z_+^2} p(n)\,  \max(1,\lambda)^{M_n} \, e^{-\beta \cdot n}.
\]
We use now the results of \cites{barany_limit_1995,vershik_limit_1994,sinai_probabilistic_1994} according to which $\log p(n)  = O(|n|^{2/3})$ and of \cite{acketa_maximal_1995} where Acketa and \Zunic have proven that $M_n = O(|n|^{2/3})$.
We will use in the sequel the additional remark that $\Zbl$ is an analytic function of $\lambda$ for all $\beta > 0$.

The partition function $Z$ is of crucial interest since its partial logarithmic derivatives are equal to expectations of macroscopic characteristics of the polygonal line. Namely, the expected coordinates of the endpoint of the line are given by:
$$
\EEbl[X_i]= \sum_{n \in \Z_+^2} \sum_{k\geq1} n_i\frac{p(n ; k)\, e^{-\beta \cdot n} \lambda^k}{Z(\beta,\lambda)}=-\frac{\partial}{\partial \beta_i} \log Z(\beta,\lambda), \qquad i \in \{1,2\}.
$$
Similarly, for $i,j \in \{1,2\}$, 
$$
\EEbl[K]=\lambda\frac{\partial}{\partial \lambda} \log Z(\beta,\lambda), \qquad \Cov_{\beta,\lambda}(X_i,X_j)=\frac{\partial^2}{\partial \beta_i\partial\beta_j} \log Z(\beta,\lambda).
$$

Taking $\lambda=1$, the probability $\PPbl$ is nothing but the two-parameter probability distribution introduced by \Sinai{}~\cite{sinai_probabilistic_1994}. Under the measure $\PPbl$, the variables $(\omega(x))_{x\in \L}$
are still independent, as in \Sinai{}'s framework,  but follow a geometric distribution only for
$\lambda=1$. In the general case, the measure \(\PPbl\) is absolutely continuous with respect to \Sinai{}'s measure with density proportional to $\lambda^{K(\cdot)}$ and the distribution of $\omega(x)$ is a biased geometric distribution. Loosely speaking, $\PPbl$ corresponds to the introduction of a {\it penalty} of the probability by a factor $\lambda$ each time a vertex appears. Strictly speaking, it is only a penalty when $\lambda < 1$ and a reward when $\lambda > 1$.

Since $\PPbl(\omega)$ depends only on the values of $X_1(\omega)$,
$X_2(\omega)$, and $K(\omega)$, we deduce
that the conditional distribution it induces on $\Omega(n_1,
n_2; k)$ is uniform. For instance, we have the following formula for all $(\beta,\lambda) \in (0,+\infty)^2\times (0,+\infty)$ which will be instrumental in the proof:
\begin{equation}
	\label{eq:Master}
    p(n_1,n_2;k)=Z(\beta,\lambda)\, e^{\beta_1 n_1} e^{\beta_2 n_2} \lambda^{-k} 
 \,\PPbl[X_1 = n_1, X_2 = n_2, K =  k].
\end{equation}

In order to get a logarithmic equivalent of $ p(n_1,n_2;k)$, our strategy is
to choose the three parameters so that 
\[
    \EEbl\left[X_1\right] = n_1,\quad
    \EEbl\left[X_2\right] = n_2,\quad
    \EEbl\left[K\right]=k.
\]
This will indeed lead to an asymptotic equivalent of $\PPbl[X_1 = n_1, X_2 = n_2, K =  k]$ due to a local limit result. This equivalent having polynomial decay, it will not interfere with the estimation of $\log p(n_1,n_2;k)$. The analysis of the partition function in the next subsection leads to a calibration of the parameters $\beta_1, \beta_2, \lambda$ satisfying the above conditions. Lemma~\ref{lem:parameters} shows that, for this calibration, 
$$
k\sim \cc(\lambda) (n_1n_2)^{1/3}
$$
and
$$
\log\Zbl\sim\beta_1 n_1\sim \beta_2 n_2 \sim \left(\frac{\zeta(3)-\Li_3(1-\lambda)}{\zeta(2)}\right)^{1/3}(n_1n_2)^{1/3}.
$$
Furthermore,  Theorem~\ref{thm:local_limit} implies that $\log\PPbl[X_1 = n_1, X_2 = n_2, K =  k]=O(\log n)$. So finally, Theorem~\ref{thm:detailed_comb} follows readily by plugging these estimates into \eqref{eq:Master}.
Note that the case $k = o(|n|^{2/3})$ corresponds to $\lambda$ going to $0$, and the above asymptotics become, as stated in Theorem~\ref{thm:detailed_comb},
\[
  \lambda \sim \frac{k^3}{n_1n_2}, \qquad \text{and}\qquad \log Z \sim \beta_1n_1 \sim \beta_2 n_2 \sim k.
\]
As a consequence, the term $\lambda^{-k}$ dominates the asymptotic in \eqref{eq:Master}, which concludes the proof.

\subsection{Estimates of the logarithmic partition function and its derivatives}

We need in the following, the analogue to the Barnes bivariate zeta function defined for $\beta=(\beta_1,\beta_2) \in (0,+\infty)^2$ by
\[
    \zeta_2^*(s ; \beta) := \sum_{x \in\L} (\beta_1 x_1 + \beta_2 x_2)^{-s},
\]
this series being convergent for \(\Re(s) > 2\). The following preliminary lemma gives useful properties of this function. This will be done by expressing this function in terms of the Barnes zeta function $\zeta_2(s,w ; \beta)$ which is defined by analytic continuation of the series
$$
\zeta_2(s,w ; \beta) = \sum_{n \in \Z_+^2} (w + \beta_1 n_1 + \beta_2 n_2)^{-s}, \qquad \Re(s) > 2, \Re(w) > 0.
$$
It is well known that $\zeta_2(s,w ; \beta)$ has a meromorphic continuation to the complex $s$-plane with simple poles at $s = 1$ and $2$, and that the residue at $s = 2$ is simply \((\beta_1\beta_2)^{-1}\). In the next lemma, we derive the relation between $\zeta_2$ and $\zeta_2^*$, and we also establish an explicit meromorphic continuation of $\zeta_2$ to the half-plane $\Re(s) > 1$ in order to obtain later polynomial bounds for $|\zeta_2^*(s)|$ as $|\Im(s)| \to +\infty$.  
Before the statement, let us recall that the fractional part $\{x\} \in [0,1)$ of a real number $x \in \RR$ is defined as $\{x\} = x - \lfloor x \rfloor$.

\begin{lemma}
\label{lem:meromorphic_continuation}
The functions $\zeta_2(s,w ; \beta)$ and $\zeta_2^*(s ; \beta)$ have a meromorphic continuation to the complex plane.
\begin{enumerate}[(i)]
    \item The meromorphic continuation of $\zeta_2(s,w ; \beta)$ to the half-plane $\Re(s) > 1$ is given by
\begin{align*}
\zeta_2(s,w ;\beta) &= \frac{1}{\beta_1\beta_2}\frac{w^{-s+2}}{(s-1)(s-2)} + \frac{(\beta_1 + \beta_2)w^{-s+1}}{2\beta_1\beta_2(s-1)} + \frac{w^{-s}}{4}\\
& - \frac{\beta_2}{\beta_1} \int_0^{+\infty} \frac{\{y\}-\frac{1}{2}}{(w+\beta_2 y)^s}\,dy - \frac{\beta_1}{\beta_2} \int_0^{+\infty} \frac{\{x\}-\frac{1}{2}}{(w+\beta_1 x)^s}\,dx\\
& - s\frac{\beta_2}{2} \int_0^{+\infty} \frac{\{y\}-\frac{1}{2}}{(w+\beta_2 y)^{s+1}}dy
- s\frac{\beta_1}{2} \int_0^{+\infty} \frac{\{x\}-\frac{1}{2}}{(w+\beta_1 x)^{s+1}}dx\\
& + s(s+1)\beta_1\beta_2 \int_0^{+\infty}\int_0^{+\infty} \frac{(\{x\} - \frac{1}{2})(\{y\}-\frac{1}{2})}{(w+\beta_1 x +\beta_2 y)^{s+2}}\,dxdy.
\end{align*}
\item The meromorphic continuation of $\zeta_2^*(s;\beta)$ is given for all $s \in \CC$ by
$$
\zeta_2^*(s;\beta) =  \frac{1}{\beta_1^s} + \frac{1}{\beta_2^s} + \frac{\zeta_2(s,\beta_1+\beta_2 ; \beta)}{\zeta(s)}.
$$
\end{enumerate}
\end{lemma}

\begin{proof}
We apply the Euler-Maclaurin formula to the partial summation defined by $F(x) = \sum_{n_2\geq0} (w +\beta_1 x + \beta_2 n_2)^{-s}$, leading to
$$
\sum_{n_1 \geq 1} F(n_1) = \int_0^{\infty} F(x)\,dx - \frac{F(0)}{2} + \int_0^\infty (\{x\}-\frac{1}{2})F'(x)\,dx.
$$
We use again the Euler-Maclaurin formula for each of the summations in $n_2$ to obtain \emph{(i)}.

In order to prove  \emph{(ii)}, we express $\zeta_2^*(s;\beta)$ in terms of $\zeta_2(s,\beta_1+\beta_2;\beta)$ for all $s$ with real part $\Re(s) > 2$. The result will follow from the analytic continuation principle. By definition of $\zeta_2^*(s;\beta)$,
$$
\zeta(s)\left[\zeta_2^*(s;\beta) - \frac{1}{\beta_1^s} - \frac{1}{\beta_2^s}\right] = \left[\sum_{d\geq1}\frac1{d^s}\right]\left[\sum_{\substack{x_1,x_2 \geq 1\\ \gcd(x_1,x_2)=1}} \frac{1}{(\beta_1 x_1 + \beta_2 x_2)^s}\right]=\sum_{x_1,x_2 \geq 1} \frac{1}{(\beta_1 x_1 + \beta_2 x_2)^s}.
$$
\end{proof}

Now we make the connection between these zeta functions and the logarithmic partition function of our modified \Sinai's model.

\begin{lemma}
\label{lem:integral_partition_function}
Let $c > 2$. For all parameters $(\beta,\lambda) \in (0,+\infty)^2 \times (0,+\infty)$,
$$
\log \Zbl = \frac{1}{2i\pi} \int_{c - i\infty}^{c + i\infty} (\zeta(s+1) - \Li_{s+1}(1-\lambda))\zeta_2^*(s;\beta)\Gamma(s)ds.
$$
\end{lemma}

\begin{proof}
Given the product form of the distribution $\PPbl$, we see that the random variables $\omega(x)$ for $x \in\L$ are mutually independent. Moreover, the marginal distribution of $\omega(x)$ is a biased geometric distribution. It is absolutely continuous with respect to the geometric distribution of parameter $e^{-\beta \cdot x}$ with density proportional to $k \mapsto \lambda^{1_{k > 0}}$. In other words, for all $k \in \Z_+$,
$$
\PPbl[\omega(x) = k] = Z_x(\beta, \lambda)^{-1} e^{-k \beta\cdot x} \lambda^{1_{k > 0}}
$$
where the normalization constant $Z_x(\beta, \lambda) = 1 + \lambda\dfrac{e^{-\beta \cdot x}}{1-e^{-\beta \cdot x}}$ is easily computed. We can now deduce the following product formula for the partition function:
\[
    \Zbl = \prod_{x\in\L} Z_x(\beta, \lambda) = \prod_{x\in\L} \left(1 + \lambda\frac{e^{-\beta \cdot x}}{1-e^{-\beta \cdot x}}\right).
\]
For now, we assume that $\lambda \in (0,1)$. Taking the logarithm of the product above
\begin{align*}
    \log \Zbl & = \sum_{x \in \Lstar} \log \left(1 + \lambda\frac{e^{-\beta \cdot x}}{1-e^{-\beta \cdot x}}\right) \\
    &= \sum_{x\in \Lstar} \log(1-(1-\lambda)e^{-\beta\cdot x}) - \sum_{x\in\Lstar} \log(1-e^{-\beta\cdot x}) \\
    &= \sum_{x\in\Lstar} \sum_{r \geq 1} \frac{1 - (1-\lambda)^r}{r} e^{-r\beta \cdot x}.
\end{align*}
Now we use the fact that the Euler gamma function $\Gamma(s)$ and the exponential function are related through Mellin's inversion formula
$$
e^{-z} = \frac{1}{2i\pi} \int_{c - i\infty}^{c + i\infty} \Gamma(s) z^{-s} ds,
$$
for all $c > 0$ and $z \in \CC$ with positive real part. Choosing $c > 2$ so that the series and the integral all converge and applying the Fubini theorem, this yields
\begin{align*}
    \log \Zbl & = \frac{1}{2i\pi} \sum_{x\in\Lstar} \sum_{r\geq 1} \int_{c - i\infty}^{c + i\infty}\frac{1-(1-\lambda)^r}{r} r^{-s} (\beta \cdot x)^{-s} \Gamma(s) ds \\
    & = \frac{1}{2i\pi} \int_{c - i\infty}^{c + i\infty} (\zeta(s+1) - \Li_{s+1}(1-\lambda))\zeta_2^*(s;\beta)\Gamma(s)\,ds.
    \end{align*}

    The lemma is proven for all $\lambda \in (0,1)$. The extension to $\lambda > 0$ will now result from analytic continuation. We already noticed that the left hand term is analytic in $\lambda$ for all fixed $\beta$. Proving the analyticity of the right hand term requires only to justify the absolute convergence of the integral on the vertical line. From Lemma~\ref{lem:meromorphic_continuation}, we know that $\zeta_2^*(c + i\tau ; \beta)$ is polynomially bounded as $|\tau|$ tends to infinity. Taking $s = c - 1 + i\tau$, successive integrations by parts of the formula
    \[
    	(\zeta(s+1) - \Li_{s+1} (1-\lambda))\Gamma(s+1) = \lambda \int_0^\infty \frac{e^x x^s}{(e^x - 1)(e^x - 1 + \lambda)} \,dx
    \]
    show for all integer $N > 0$, there exists a constant $C_N > 0$ such that, uniformly in $\tau$, 
    \begin{equation}
    	\label{eq:decay_lambda}
    	\bigl | (\zeta(s+1) - \Li_{s+1} (1-\lambda))\Gamma(s+1) \bigr | \leq \frac{ C_N\lambda}{(1+|\tau|)^N}.
    \end{equation}
\end{proof}

Finally, the next Lemma makes use of the contour integral representation of $\log Z(\beta,\lambda)$ to derive at the same time an asymptotic formula for each one of its derivatives.

\begin{lemma}
\label{lem:derivatives_Z}
Let $(p,q_1,q_2) \in \ZZ_+^3$. For all $\epsilon > 0$, there exists $C > 0$ such that
$$
\left|\left[\lambda\frac{\partial}{\partial \lambda}\right]^p\left[\frac{\partial}{\partial \beta_1}\right]^{q_1}\left[\frac{\partial}{\partial \beta_2}\right]^{q_2} \left( \log \Zbl - \frac{\zeta(3)-\Li_3(1-\lambda)}{\zeta(2)\beta_1\beta_2}\right)\right| \leq \frac{C\,\lambda}{|\beta|^\kappa}
$$
with $\kappa = q_1 + q_2 + 1 + \epsilon$, uniformly in the region $\{(\beta,\lambda) \mid \epsilon < \frac{\beta_1}{\beta_2} < \frac{1}{\epsilon} \text{ and } 0 < \lambda < \frac{1}{\epsilon}\}$. 
\end{lemma}

\begin{proof}
    Lemma~\ref{lem:integral_partition_function} provides an integral representation of the logarithmic partition function $\log\Zbl$. We will use the residue theorem to shift the contour of integration from the vertical line $\Re(s) = 3$ to the line $\Re(s) = 1 + \epsilon$. Lemma~\ref{lem:meromorphic_continuation} shows that the function $M(s) := (\zeta(s+1)-\Li_{s+1}(1-\lambda)))\zeta_2^*(s;\beta)\Gamma(s)$ is meromorphic in the strip $1 < \Re(s) < 3$ with a single pole at $s=2$, where the residue is given by
$$
\frac{\zeta(3)-\Li_3(1-\lambda)}{\zeta(2)} \cdot \frac{1}{\beta_1\beta_2}
$$ 
From the inequality~\eqref{eq:decay_lambda}, Lemma~\ref{lem:meromorphic_continuation} and the fact that $|\zeta(s)|$ has no zero with $\Re(s) > 1$, we see that $M(s)$ vanishes uniformly in $1 + \epsilon \leq \Re(s) \leq 3$ when $|\Im(s)|$ tends to $+ \infty$. By the residue theorem,
\begin{equation}
\label{eq:residue_theorem}
\log \Zbl= \frac{\zeta(3)-\Li_3(1-\lambda))}{\zeta(2)\beta_1\beta_2}+ \frac{1}{2i\pi} \int_{1 + \epsilon - i\infty}^{1 + \epsilon + i\infty} M(s)\,ds.
\end{equation}
From the Leibniz rule applied in the formula of Lemma~\ref{lem:meromorphic_continuation}~(i), we obtain directly the meromorphic continuation of $\frac{\partial^{q_1}}{\partial \beta_1^{q_1}}\frac{\partial^{q_2}}{\partial \beta_2^{q_2}}\zeta_2(s,\beta_1+\beta_2;\beta)$ in the half-plane $\Re(s) > 1$. We also obtain the existence of a constant $C > 0$ such that
\[
    \left|\left[\frac{\partial}{\partial \beta_1}\right]^{q_1}\left[\frac{\partial}{\partial \beta_2}\right]^{q_2}\zeta_2(1 + \epsilon + i\tau ,\beta_1+\beta_2;\beta)\right| \leq  \frac{C\, |\tau|^{2 + q_1 + q_2}}{|\beta|^\kappa}
\]
with $\kappa = q_1 + q_2 + 1 + \epsilon$.
A reasoning similar to the one we have used in order to derive \eqref{eq:decay_lambda} shows that for all integers $p$ and $N > 0$, there exists a constant $C_{p,N}$ such that, uniformly in $\tau$, 
    \[
    	\left | \left[\lambda \frac{\partial}{\partial\lambda}\right]^p (\zeta(s+1) - \Li_{s+1} (1-\lambda))\Gamma(s+1) \right | \leq \frac{ C_{p,N}\,\lambda}{(1+|\tau|)^N}.
    \]

    In order to differentiate both sides of equation~\eqref{eq:residue_theorem} and permute the partial derivatives and the integral sign, we have to mention the fact that the Riemann zeta function is bounded from below on the line $\Re(s) = 1 + \epsilon$ and that the derivatives of $\Li_s(1-\lambda)$ with respect to $\lambda$ are all bounded. This also gives the announced bound on the error term.
\end{proof}

\subsection{Calibration of the shape parameters}
\label{sec_conv_calibration}

When governed by the Gibbs measure \(\PPbl\), the expected value of the random vector with components
\[
    X_1(\omega) = \sum_{x\in\L} \omega(x) x_1, \quad
    X_2(\omega) = \sum_{x\in\L} \omega(x) x_2, \quad
    K(\omega) = \sum_{x\in\L} \indicator{\omega(x) > 0},
\]
is simply given by the logarithmic derivatives of the partition function \(\Zbl\). Remember that we planned to choose \( \lambda \) and  \( \beta_1, \beta_2 \) as functions of \(n = (n_1,n_2)\) an \(k\) in order for the probability $\PP[X_1 = n_1, X_2 = n_2, K = k]$ to be maximal, which is equivalent to $\EE(X_1)=n_1$, $\EE(X_2)=N_2$ and $\EE(K)=k$. We address this question in the next lemma.

\begin{lemma}
    \label{lem:parameters}
    Assume that $n_1,n_2,k$ tend to infinity with $n_1 \asymp n_2$ and $|k| = O(|n|^{2/3})$. There exists a unique choice of $(\beta_1,\beta_2,\lambda)$ as functions of $(n,k)$ such that
    \[
        \EEbl[X_1] = n_1, \quad \EEbl[X_2] = n_2, \quad \EEbl[K] = k.
    \]
    Moreover, they satisfy
    \begin{equation}
        \label{eq:conv_parameters}
        n_1 \sim \frac{\zeta(3) - \Li_3(1-\lambda)}{\zeta(2)({\beta_1})^2{\beta_2}}, \quad n_2 \sim \frac{\zeta(3)-\Li_3(1-\lambda)}{\zeta(2){\beta_1}({\beta_2})^2}, \quad
                k \sim -\frac{\lambda\partial_\lambda
             \Li_3(1-\lambda)}{\zeta(2){\beta_1}{\beta_2}}.
     \end{equation}
    If $k = o(|n|^{2/3})$, then $\lambda$ goes to $0$ and the above relations yield
    \[
        \beta_1 \sim \frac{k}{n_1}, \quad \beta_2 \sim \frac{k}{n_2}, \quad \lambda \sim \frac{k^3}{n_1n_2}.
    \]
\end{lemma}

\begin{proof}
    With the change of variable $\lambda = e^{-\gamma}$, the existence and uniqueness of $(\beta,\lambda)$ are equivalent to the fact that the function
\[
f \colon (\beta_1,\beta_2,\gamma) \mapsto \beta_1 n_1 + \beta_2 n_2 + \gamma k + \log Z(\beta, e^{-\gamma})
\]
has a unique critical point in the open domain $D = (0,+\infty)^2 \times \RR$.
First observe that $f$ is smooth and strictly convex since its Hessian matrix is actually the covariance matrix of the random vector $(X_1,X_2,K)$. In addition, from the very definition~\eqref{eq:Zbl} of $\Zbl$, we can see that $f$ converges to $+\infty$ in the neighborhood of any point of the boundary of $D$ as well as when $|\beta_1| + |\beta_2| + |\gamma|$ tends to $+\infty$. The function being continuous in $D$, this implies the existence of a minimum, which by convexity is the unique critical point $(\beta^*,\gamma^*)$ of $f$.

From now on, we will be concerned and check along the proof that we stay in the regime $\beta_1,\beta_2\to 0$, $\beta_1 \asymp \beta_2$, and $\gamma$ bounded from below. From Lemma~\ref{lem:derivatives_Z}, we can approximate $f$ by the simpler function
\[
    g \colon (\beta_1,\beta_2,\gamma) \mapsto \beta_1 n_1 + \beta_2 n_2 + \gamma k + \dfrac{\zeta(3)-\Li_3(1-e^{-\gamma})}{\beta_1\beta_2}
\]
with $|f(\beta,\gamma) - g(\beta,\gamma)| \leq \dfrac{Ce^{-\gamma}}{|\beta|^{3/2}}$ for some constant $C > 0$. The unique critical point $(\tilde{\beta},\tilde{\gamma})$ of $g$ satisfies
\[
 n_1 = \frac{\zeta(3) - \Li_3(1-e^{-\tilde{\gamma}})}{\zeta(2)(\tilde{\beta_1})^2\tilde{\beta_2}},
 \quad n_2 = \frac{\zeta(3)-\Li_3(1-e^{-\tilde{\gamma}})}{\zeta(2)\tilde{\beta_1}(\tilde{\beta_2})^2}, \quad 
   k = -\frac{e^{-\tilde{\gamma}}\partial_\lambda \Li_3(1-e^{-\tilde{\gamma}})}{\zeta(2)\tilde{\beta_1}\tilde{\beta_2}}.
\]
The goal now is to prove that $(\beta^*,\gamma^*)$ is close to $(\tilde{\beta},\tilde{\gamma})$. To this aim, we find a convex neighborhood $C$ of $(\tilde{\beta},\tilde{\gamma})$ such that $g|_{\partial C} \geq g(\tilde{\beta},\tilde{\gamma}) + \frac{Ce^{-\tilde{\gamma}}}{\tilde{\beta_1}\tilde{\beta_2}}$ .
In the neighborhood of $(\tilde{\beta},\tilde{\gamma})$ the expression of the Hessian matrix of $g$ yields $g(\tilde{\beta}_1 + t_1, \tilde{\beta}_2 + t_2, \tilde{\gamma} + u) \geq g(\tilde{\beta}_1 , \tilde{\beta}_2 , \tilde{\gamma} ) + \frac{\tilde{C}e^{-\tilde{\gamma}} }{(\tilde{\beta}_1\tilde{\beta}_2)^2}(\|t\|^2 + \tilde{\beta}_1\tilde{\beta}_2 |u|^2)$. Therefore we need only take
\[
    C = [\tilde{\beta}_1 - C_1\tilde{\beta}_1^{5/4}, \tilde{\beta}_1 + C_1\tilde{\beta}_1^{5/4}] \times  
    [\tilde{\beta}_2 - C_2\tilde{\beta}_2^{5/4}, \tilde{\beta}_2 + C_2\tilde{\beta}_2^{5/4}] \times [\tilde{\gamma} - C_3 |\beta|^{1/4}, \tilde{\gamma} + C_3 |\beta|^{1/4}].
\]
Therefore, $f|_{\partial C} > f(\tilde{\beta},\tilde{\gamma})$. By convexity of $f$ and $C$ this implies $(\beta^*,\gamma^*) \in C$. Hence
\[
    \beta_1^* \sim \tilde{\beta}_1, \quad
    \beta_2^* \sim \tilde{\beta}_2, \quad
    e^{-\gamma^*} \sim e^{-\tilde{\gamma}},
\]
concluding the proof.
\end{proof}

\subsection{A local limit theorem}

In this section, we show that the random vector $(X_1,X_2,K)$ satisfies a
local limit theorem when the parameters are calibrated as above. Let $\Gammabl$ be the covariance matrix under the measure $\PPbl$ of the random vector $(X_1,X_2, K)$.
\begin{theorem}[Local limit theorem]
    \label{thm:local_limit}
    Let us assume that $n_1,n_2,k$ tend to infinity such that $n_1 \asymp n_2 \asymp |n|$, $\log |n| = o(k)$, and $k = O(|n|^{2/3})$. For the choice of parameters made in Lemma~\ref{lem:parameters},
    \begin{equation}
        \label{eq:local_limit}
        \PPbl[X = n, K = k] \sim \frac{1}{(2\pi)^{3/2}}\frac{1}{\sqrt{\det \Gammabl}}.
    \end{equation}
Moreover,
\begin{equation}
    \det \Gammabl \asymp \frac{|n|^4}{k}
\end{equation}
If $k = o(|n|^{2/3})$,
    \begin{equation}
        \PPbl[X = n, K = k] \sim \frac{1}{(2\pi)^{3/2}}\frac{\sqrt{k}}{n_1n_2}
    \end{equation}
\end{theorem}
 
This result is actually an application of a more general lemma proven by the first author in \cite[Proposition 7.1]{bureaux_partitions_2014}. In order to state the lemma, we introduce some notations. Let $\sigmabl^2$ be the smallest eigenvalue of $\Gammabl$.  
Introducing $X_{1,x} = \omega(x) \cdot x_1$, $X_{2,x} = \omega(x) \cdot x_2$ and $K_x = 1_{\{\omega(x) > 0\}}$ as well as $\overline{X_{1,x}} , \overline{X_{2,x}},\overline{K_x}$ their centered counterparts, let $\Lbl$ be the Lyapunov ratio 
\[
\Lbl := \sup_{(t_1,t_2,u)\in \RR^3}  \sum_{x \in \L} \frac{\EEbl \left\lvert t_1 \overline{X_{1,x}} + t_2 \overline{X_{2,x}} +  u\overline{K_x}\right\rvert^3}{\Gammabl(t_1,t_2,u)^{3/2}}.
\]where $\Gammabl(\cdot)$ stands for the quadratic form canonically associated to $\Gammabl$.
Let $ \phibl(t,u) = \EEbl (e^{i(t_1 X_1 + t_2 X_2 + uK})$ for all $(t_1,t_2,u) \in \RR^3$. Finally, we consider the ellipsoid $\Ecal_{\beta,\lambda}$ defined by
\[
\mathcal{E}_{\beta,\lambda} := \left\{(t_1,t_2,u) \in \RR^3 \mid \Gammabl(t_1,t_2,u) \leq (4\Lbl)^{-2}\right\}.
\]

The following lemma is a reformulation of Proposition~{7.1} in \cite{bureaux_partitions_2014}. It gives three conditions on the product distributions $\PPbl$ that entail a local limit theorem with given speed of convergence.

 \begin{lemma}
     \label{lem:framework}
     With the notations introduced above, suppose that there exists a family of number $(a_{\beta,\lambda})$ such that
     \begin{gather}
     \frac{1}{\sigmabl\sqrt{\det \Gammabl}} = \bigO(a_{\beta,\lambda}), \\
              \frac{\Lbl}{\sqrt{\det \Gammabl}} = \bigO(a_{\beta,\lambda}), \\
              \label{eq:cramer}
               \sup_{(t,u) \in [-\pi,\pi]^3\setminus \Ecal_{\beta,\lambda}} \left|\phibl(t,u)\right| = \bigO(a_{\beta,\lambda}).
     \end{gather}
     Then, a local limit theorem holds uniformly for $\PPbl$ with rate $a_{\beta,\lambda}$:
     \[
         \sup_{(n,k)\in \Z^3}\; \left|\PPbl[X = n, K = k] - \frac{\exp\left[-\frac{1}{2}\Gammabl^{-1}\bigl((n,k) - \EEbl (X,K)\bigr)\right]}{(2\pi)^{3/2}\sqrt{\det \Gammabl}}\right| = \bigO(a_{\beta,\lambda}).
     \]
 \end{lemma}
 
 When governed by the Gibbs measure \(\PPbl\), the covariance matrix $\Gammabl$ of the random vector \((X_1,X_2,K)\) is simply given by the Hessian matrix of the log partition function $\log \Zbl$. Let \(u(\lambda) := (\zeta(3) - \Li_3(1-\lambda))/\zeta(2)\) for \(\lambda > 0\). Applications of Lemma~\ref{lem:derivatives_Z} for all $(p,q_1,q_2) \in \ZZ_+^3$ such that $p+q_1+q_2 = 2$ imply that this covariance matrix is asymptotically equivalent to
\[
    \begin{bmatrix}
        \beta_1\beta_2 & 0 & 0\\
        0 & \beta_1^3 \beta_2 & 0\\
        0 & 0 & \beta_1\beta_2^3
    \end{bmatrix}^{-\frac{1}{2}}
    \begin{bmatrix}
        \lambda^2 u''(\lambda) + \lambda u'(\lambda)  &
        \lambda u'(\lambda) &
        \lambda u'(\lambda) \\

        \lambda u'(\lambda) &
        2u(\lambda) &
        u(\lambda) \\

        \lambda u'(\lambda) &
        u(\lambda) &
        2u(\lambda) &
    \end{bmatrix}
    \begin{bmatrix}
        \beta_1\beta_2 & 0 & 0\\
        0 & \beta_1^3 \beta_2 & 0\\
        0 & 0 & \beta_1\beta_2^3
    \end{bmatrix}^{-\frac{1}{2}}.
\]
A straightforward calculation shows that this matrix is positive definite for all $\lambda > 0$.

\begin{lemma}
\label{lem:covariance}
The random vector \((X_1,X_2, K)\) has a covariance matrix \(\Gammabl\) satisfying
\[
\Gammabl(t,u) \asymp \frac{(n_1)^{5/3}}{(\lambda n_2)^{1/3}} |t_1|^2 + \frac{(n_2)^{5/3}}{(\lambda n_1)^{1/3}} |t_2|^2 + (\lambda n_1n_2)^{1/3} |u|^2, \qquad |n| \to +\infty.
\]
\end{lemma}

\begin{proof}
All the coefficients of the previous matrix $u(\lambda), \lambda u'(\lambda), \lambda^2 u''(\lambda)$ are of order $\lambda$ in the neighborhood of $0$, and the determinant is equivalent to $\lambda^3$. Therefore, the eigenvalues are also of order $\lambda$. The result follows from the fact that the values of $\beta_1$ and $\beta_2$ are given by~\eqref{eq:conv_parameters} and that $\zeta(3)-\Li_3(1-\lambda) \asymp \zeta(2) \lambda$.
\end{proof}

\begin{lemma}
    \label{lem:lyapunov}
    The Lyapunov coefficient satisfies \(\Lbl = O(\lambda^{-1/6} \lvert n \rvert^{-1/3})\).
\end{lemma}

\begin{proof}
Using Lemma~\ref{lem:covariance}, there exists a constant $C > 0$ such that
\[
\Lbl \leq C \sum_{x \in \XX} \left[\frac{\EEbl\lvert \overline{X_{1,x}}  \rvert^3 }{\lambda^{-1/2}}\frac{n_2^{1/2}}{n_1^{5/2}} + \frac{\EEbl\lvert \overline{X_{2,x}}  \rvert^3 }{\lambda^{-1/2}}\frac{n_1^{1/2}}{n_2^{5/2}} + \frac{\EEbl\lvert \overline{K_x}  \rvert^3 }{\lambda^{1/2}(n_1n_2)^{1/2}}\right].
\]
Therefore, we need only prove that
\[
 \sum_{x \in \L} \EEbl \left\lvert \overline{K_x}  \right\rvert^3 = O(\lvert n\rvert^{2/3}), \qquad
 \sum_{x \in \L} \EEbl \left\lvert \overline{X_{i,x}} \right\rvert^3 = O(\lvert n\rvert^{5/3}).
\]
Notice that for a Bernoulli random variable $B(p)$ of parameter $p$, one has $\EE[\lvert B(p) - p\rvert^3] \leq 4 (\EE [B(p)^3] + p^3) \leq 8 p$. This implies
\[
	 \sum_{x \in \L} \EEbl \left\lvert \overline{K_x}  \right\rvert^3 \leq \sum_{x \in \L} \frac{8 \lambda e^{-\beta\cdot x}}{1 - (1-\lambda) e^{-\beta\cdot x}} \leq \sum_{x \in \L} \frac{8 \lambda e^{-\beta\cdot x}}{1 - e^{-\beta\cdot x}} = O(\frac{\lambda}{\beta_1\beta_2}).
\]
Similarly, we obtain
\[	\sum_{x \in \L} \EEbl \left\lvert \overline{X_{1,x}} \right\rvert^3 = O(\frac{\lambda}{\beta_1^4\beta_2}),
	\quad
	\sum_{x \in \L} \EEbl \left\lvert \overline{X_{2,x}} \right\rvert^3 = O(\frac{\lambda}{\beta_1\beta_2^4}).
\]
\end{proof}

\begin{lemma}
    \label{lem:cramer}
    Condition~\eqref{eq:cramer} of Lemma~\ref{lem:framework} is satisfied.
    More precisely,
    \[
        \limsup_{|n| \to +\infty} \quad\sup_{(t,u) \in [-\pi,\pi]^3\setminus\mathcal{E}_{\beta,\lambda}}\quad\frac{1}{\lambda^{1/3} |n|^{2/3}} \log |\phi_n(t,u)| < 0.
    \]
\end{lemma}

\begin{proof}
    From Lemmas~\ref{lem:covariance} and~\ref{lem:lyapunov}, there exists a constant \(c > 0\) depending on \(\lambda\) such that for all \(n=(n_1,n_2)\) with \(|n|\) large enough,
    \[
        [-\pi,\pi]^3 \setminus \mathcal{E}_{\lambda,n} \subset
        \{(t,u) \in \RR^3 \mid c < |u| \leq \pi \text{ or } c \lambda^{1/3} |n|^{-1/3} < |t|\}.
    \]
    The strategy of the proof is to deal separately with the cases $|u| > c$ and $|t| > c \lambda^{1/3} |n|^{-1/3}$, which requires to find first adequate bounds for $|\phi_n(t,u)|$ in both cases. For all \((t_1,t_2, u) \in \RR^3\) and \(x\in \L\), let us write \(t = (t_1,t_2)\) and \(\rho^x = e^{-\beta \cdot x}\). The ``partial'' characteristic function \(\phi_n^x(t,u) = \EE[e^{i(t\cdot X_x + uK_x )}]\) is given by
    \[
        \phi_n^x(t,u) = \left(1 + \lambda e^{iu}\dfrac{e^{it\cdot x}\rho^x}{1-  e^{it\cdot x}\rho^x}\right)\left(1 + \lambda\dfrac{\rho^x}{1-  \rho^x}\right)^{-1},
    \]
    hence a straightforward calculation yields
    \begin{align*}
        \left|\phi_n^x(t,u)\right|^2 %&= \frac{1+\frac{4\rho^x}{(1-(1-\lambda)\rho^x)^2}\left[|\sin(\frac{t\cdot x}{2})|^2 - \lambda |\sin(\frac{t\cdot x + u}{2})|^2 + \lambda\rho^x |\sin(\frac{u}{2})|^2 \right]}{1+\frac{4\rho^x}{(1-\rho^x)^2}|\sin(\frac{t\cdot x}{2})|^2}\\
                                     & = 1 - \frac{\frac{4\lambda\rho^x}{(1-(1-\lambda)\rho^x)^2}\left[\frac{\rho^x(2+(\lambda-2)\rho^x)}{(1-\rho^x)^2}|\sin(\frac{t\cdot x}{2})|^2 + |\sin(\frac{t\cdot x + u}{2})|^2 - \rho^x |\sin(\frac{u}{2})|^2 \right]}{1+\frac{4\rho^x}{(1-\rho^x)^2}|\sin(\frac{t\cdot x}{2})|^2}\\
                                     & \leq \exp\left\{ - \frac{\frac{4\lambda\rho^x}{(1-(1-\lambda)\rho^x)^2}\bigl(2\rho^x|\sin(\frac{t\cdot x}{2})|^2 + |\sin(\frac{t\cdot x + u}{2})|^2 - \rho^x |\sin(\frac{u}{2})|^2 \bigr)}{1+\frac{4\rho^x}{(1-\rho^x)^2}|\sin(\frac{t\cdot x}{2})|^2}\right\}
    \end{align*}
    Using the law of sines in a triangle with angles $\frac{t\cdot x}{2}$, $\frac{u}{2}$ and $\frac{2\pi - t\cdot x+u}{2}$, we see that the numerator inside the bracket is proportional (with positive constant) to
    \[
        2\rho^x \|a\|^2 + \|{b}\|^2 - \rho^x \|{a} + {b}\|^2
    \]
    where $a$ and $b$ are two-dimensional vectors. Since the real quadratic form
    $(a_i,b_i) \mapsto 2\rho\, a_i^2 + b_i^2 - \frac{2\rho}{1+2\rho} \,(a_i + b_i)^2$
    is positive for all $\rho \in (0,1)$ and for $i \in \{1,2\}$, we deduce that 
    \begin{equation}
        \label{eq:sinu}
        |\phi_n^x(t,u)| \leq
        \exp\left\{ - \frac{\frac{2\lambda\rho^x}{(1-(1-\lambda)\rho^x)^2}}{1+\frac{4\rho^x}{(1-\rho^x)^2}}\left(\frac{2\rho^x}{1+2\rho^x}-\rho^x\right)\left|\sin(\tfrac{u}{2})\right|^2 \right\}
    \end{equation}
    for all $x$ such that $\rho_x \leq \frac{1}{2}$.
    In the same way, the positivity of the quadratic form $(a_i,b_i)  \mapsto \frac{\rho}{1-\rho}\, a_i^2 + b_i^2 - \rho\, (a_i+b_i)^2$ yields
    \begin{equation}
        \label{eq:sintx}
        |\phi_n^x(t,u)| \leq
        \exp\left\{ - \frac{\frac{2\lambda\rho^x}{(1-(1-\lambda)\rho^x)^2}}{1+\frac{4\rho^x}{(1-\rho^x)^2}}\left(2\rho^x-\frac{\rho^x}{1-\rho^x}\right)\left|\sin(\tfrac{t\cdot x}{2})\right|^2 \right\}
    \end{equation}
    for all $x$ such that $\rho_x \leq \frac{1}{2}$.

    Let us begin with the region \(\{(t,u) \in \RR^3 \mid c < |u| \leq \pi\}\). In this case $|\sin(\tfrac{u}{2})|$ is uniformly bounded from below by $|\sin(\tfrac{c}{2})|$. Hence using \eqref{eq:sinu} for the $x \in \XX$ such that $\frac{1}{4} < \rho^x \leq \frac{1}{3}$ and the bound $|\phi_n^x(t,u)| \leq 1$ for all other $x$, we obtain
    \[
        \log |\phi_n(t,u)| \leq - \frac{1}{160}\frac{\lambda|\sin(\tfrac{c}{2})|^2}{(1+\frac{1}{3}|\lambda-1|)^2}\left|\left\{x \in \XX \mid \frac{1}{4} < \rho^x \leq \frac{1}{3}\right\}\right|.
    \]
    To conclude, let us recall that the number of integral points with coprime coordinates such that $\frac{1}{4} < e^{-\beta \cdot x} \leq \frac{1}{3}$ is asymptotically equal to $\frac{1}{\zeta(2)}\frac{\log(4/3)}{2\beta_1\beta_2} \asymp \lambda^{-2/3} |n|^{2/3}$.
    
    We now turn to the region \(\{(t,u) \in [-\pi,\pi]^3 \mid c\lambda^{1/3} |n|^{-1/3} < |t|\}\). Without loss of generality, we can assume \(|t_1| > c' \lambda^{1/3} |n|^{-1/3}\) for some universal constant \(c' \in (0; c)\). Using the inequality \eqref{eq:sintx} for the elements $x \in \XX$ such that $\frac{1}{4} < \rho^x \leq \frac{1}{3}$ and the bound $|\phi_n^x(t,u)| \leq 1$ for all other $x$, we obtain for all $\epsilon \in (0,1)$,
    \[
        \log |\phi_n(t,u)| \leq - \frac{\epsilon^2}{64} \frac{\lambda}{(1 + \frac{1}{3}|\lambda-1|)^2} \left|\left\{ x \in \XX \mid \frac{1}{4} < \rho^x \leq \frac{1}{3} \text{ and } |\sin(\tfrac{t \cdot x}{2})| \geq \epsilon\right\}\right|.
    \]
  Since the number of $x \in \XX$ such that $\frac{1}{4} < e^{-\beta \cdot x} \leq \frac{1}{3}$ is asymptotically equal to $\frac{\log(4/3)}{2\zeta(2)\beta_1\beta_2}$, it is enough to prove that we can find $\epsilon$ such that the set of vectors $x \in \ZZ_+^2$ with $|\sin(\tfrac{t\cdot x}{2})| < \epsilon$ has density strictly smaller than $\frac{1}{\zeta(2)}$ in $\{x \in \ZZ_+^2 \mid \frac{1}{4} < \rho^x \leq \frac{1}{3}\}$.
  We split up this region according to horizontal lines, that is to say with $\frac{t_2x_2}{2}$ constant. The set $\{x_1 \in \RR \mid |\sin(\frac{t_2x_2}{2}+\tfrac{t_1x_1}{2})| < \epsilon\}$ is a periodic union of strips of period $\tau_1 = \frac{2\pi}{t_1} \geq 2$ and width bounded by $4\epsilon\tau_1$. Hence the number of $x_1 \in \ZZ_+$ satisfying this condition and lying in any bounded finite interval $I$ is at most $\left(\frac{|I|}{\tau_1} + 2\right)(4\epsilon \tau_1+1)$. Summing up the contributions of the horizontal lines, this shows the existence of some positive constant $C > 0$ independent of $\epsilon$ such that for all $\epsilon \in (0,1)$, the number of $x \in \ZZ_+^2$ satisfying both $\frac{1}{4} < e^{-\beta \cdot x} \leq \frac{1}{3}$ and $|\sin(\tfrac{t\cdot x}{2})| < \epsilon$ is bounded by
  \[
      (\tfrac{1}{2} + C\epsilon) \frac{\log(4/3)}{2\beta_1\beta_2} + C |n|^{1/3}\log|n|.
  \]
  To achieve our goal, we can therefore choose $\epsilon = \frac{1}{2C}(\frac{1}{\zeta(2)}-\frac{1}{2}) > 0$.
\end{proof}

\begin{proof}[Proof of Theorem~\ref{thm:local_limit}]
    We simply check that the hypotheses of Lemma~\ref{lem:framework} are satisfied. From Lemma~\ref{lem:covariance}, we have $\sigmabl^2 \asymp k$ and $\det(\Gammabl) \asymp k^{-1}|n|^4$, hence
   \[
       \frac{1}{\sigmabl \sqrt{\det \Gammabl}} \asymp \frac{1}{|n|^2}.
   \]
   Using in addition Lemma~\ref{lem:lyapunov}, we have also
   \[
       \frac{\Lbl}{\sqrt{\det \Gammabl}} = O\left(\frac{1}{|n|^2}\right).
   \]
   Finally, Lemma~\ref{lem:cramer} shows the existence of some constant $c > 0$ such that for all $(n,k)$ large enough,
   \[
        \sup_{(t,u) \in [-\pi,\pi]^3
            \setminus\mathcal{E}_{\beta,\lambda}}
        \; |\phi_n(t,u)| \leq e^{-c k}
    \]
    Since we have made the assumption $\log |n| = o(k)$, the quantity $e^{-ck}$ is also bounded from above by $|n|^{-2}$. Therefore, all hypotheses of Lemma~\ref{lem:framework} are satisfied. As a consequence, $\PPbl$ satisfies a local limit theorem with speed rate $a_{\beta,\lambda} \asymp |n|^{-2}$. 

\end{proof}

\section{Limit shape}
\label{sec:limit_shape}

We start by proving the existence of a limit shape in the modified \Sinai{} model, which is the aim of the next two lemmas. The natural normalization for the convex polygonal line is to divide each coordinate by the corresponding expectations for the final point.

The first lemma shows that the arc of parabola is the limiting curve of the expectation of the random convex polygonal line $m_i^\theta (\beta,\lambda) = \EEbl [X_i^\theta]$ for $i\in \{1,2\}, \theta \in [0,\infty]$ under the $\PPbl$ distribution.

\begin{lemma}
    \label{lem:expectation}
    Suppose that $\beta_1$ and $\beta_2$ tend to $0$ such that $\beta_1 \asymp \beta_2$ and $\lambda$ is bounded from above. Then
    \[
        \lim_{|\beta| \to 0} \sup_{\theta \in [0,\infty]} \left |
        \left[\frac{m_1^\theta(\beta,\lambda)}{m_1^\infty(\beta,\lambda)},
        \frac{m_2^\theta(\beta,\lambda)}{m_2^\infty(\beta,\lambda)}\right] - 
        \left[
            \frac{\theta(\theta + 2\frac{\beta_1}{\beta_2})}{(\theta + \frac{\beta_1}{\beta_2})^2} 
            ,
            \frac{\theta^2}{(\theta + \frac{\beta_1}{\beta_2})^2}
        \right] \right| = 0.
    \]
\end{lemma}

\begin{proof}
    Since we are dealing with continuous increasing functions, the uniform convergence convergence will follow from the simple convergence.
    We mimic the proof of Lemma~\ref{lem:derivatives_Z}, except that the domain of summation $\XX$ is replaced by the subset of vectors $x$ such that $x_2 \leq \theta x_1$. The expectations are given by the first derivatives of the \emph{partial} logarithmic partition function
    \[
        \log Z^\theta(\beta, \lambda) = \frac{1}{2i\pi}\int_{c-i\infty}^{c+i\infty} (\zeta(s+1) - \Li_{s+1}(1-\lambda)) {\zeta_2^{\theta,*}}(s)\Gamma(s)\,ds
    \]
    where $\zeta_2^{\theta,*}$ is the restricted zeta function defined by analytic continuation of the series
    \begin{align*}
        \zeta_2^{\theta,*}(s) & = \sum_{\substack{x \in \XX \\ x_2 \leq \theta x_1}} (\beta_1 x_1 + \beta_2 x_2)^{-s} \\
                              & = \frac{1}{\beta_1^s} + \frac{1_{\{\theta = \infty\}} }{\beta_2^s} + \frac{1}{\zeta(s)}\sum_{\substack{x_1,x_2 \geq 1\\ x_2 \leq \theta x_1}} (\beta_1 x_1 + \beta_2 x_2)^{-s}.
    \end{align*}

    The continuation of the underlying restricted Barnes zeta function is obtained using the Euler-Maclaurin formula several times:
    \begin{align*}
        \sum_{x_2 = 1}^{\lfloor \theta x_1 \rfloor} (\beta_1 x_1 + \beta_2 x_2)^{-s} & =
        \int_1^{\lfloor \theta x_1 \rfloor} (\beta_1 x_1 + \beta_2 x_2)^{-s}\,dx_2  + \frac{(\beta_1 x_1 + \beta_2)^{-s}}{2} + \frac{(\beta_1 x_1 + \beta_2 \lfloor \theta x_1 \rfloor)^{-s}}{2} \\
        & \qquad -s \beta_2 \int_1^{\lfloor \theta x_1 \rfloor}(\{x_2\} - \frac{1}{2}) (\beta_1 x_1 + \beta_2 x_2)^{-(s+1)}\,dx_2\\
        & = \int_1^{\theta x_1} (\beta_1 x_1 + \beta_2 x_2)^{-s}\,dx_2  + \frac{(\beta_1 x_1 + \beta_2)^{-s}}{2} + \frac{(\beta_1 x_1 + \beta_2 \lfloor \theta x_1 \rfloor)^{-s}}{2} \\
        & \qquad -s \beta_2 \int_1^{\lfloor \theta x_1 \rfloor}(\{x_2\} - \frac{1}{2}) (\beta_1 x_1 + \beta_2 x_2)^{-(s+1)}\,dx_2 \\
        & \qquad- \int_{\lfloor \theta x_1 \rfloor}^{\theta x_1} (\beta_1 x_1 + \beta_2 x_2)^{-s}\,dx_2\\
        %& = \int_1^{\theta x_1} (\beta_1 x_1 + \beta_2 x_2)^{-s}\,dx_2  + \frac{(\beta_1 x_1 + \beta_2)^{-s}}{2} + \frac{(\beta_1 x_1 + \beta_2 \lfloor \theta x_1 \rfloor)^{-s}}{2} \\
        %& \qquad -s \beta_2 \int_1^{\lfloor \theta x_1 \rfloor}(\{x_2\} - \frac{1}{2}) (\beta_1 x_1 + \beta_2 x_2)^{-(s+1)}\,dx_2 \\
        %& \qquad- \int_{\lfloor \theta x_1 \rfloor}^{\theta x_1} (\beta_1 x_1 + \beta_2 x_2)^{-s}\,dx_2\\
        & = \frac{(\beta_1 x_1 + \beta_2)^{-s+1}}{\beta_2(s-1)} 
        - \frac{(\beta_1 x_1 + \beta_2 \theta x_1)^{-s+1}}{\beta_2(s-1)}  + R(s,x_1,\beta_1,\beta_2,\theta)
    \end{align*}
    where
    \begin{align*}
        R(s,x_1,\beta_1,\beta_2,\theta) & = \frac{(\beta_1 x_1 + \beta_2)^{-s}}{2} + \frac{(\beta_1 x_1 + \beta_2 \lfloor \theta x_1 \rfloor)^{-s}}{2} \\
        & \qquad -s \beta_2 \int_1^{\lfloor \theta x_1 \rfloor}(\{x_2\} - \frac{1}{2}) (\beta_1 x_1 + \beta_2 x_2)^{-(s+1)}\,dx_2 \\
        & \qquad- \int_{\lfloor \theta x_1 \rfloor}^{\theta x_1} (\beta_1 x_1 + \beta_2 x_2)^{-s}\,dx_2\\
    \end{align*}
    is such that $\sum_{x_1 \geq 1} R(s,x_1,\beta_1,\beta_2,\theta)$ converges absolutely for all $s$ with $\Re(s) > 1$. Therefore the latter series defines a holomorphic function in the half-plane $\Re(s) > 1$. Finally,
    \begin{align*}
        \sum_{\substack{x_1,x_2 \geq 1\\ x_2 \leq \theta x_1}} (\beta_1 x_1 + \beta_2 x_2)^{-s}
            & = \frac{(\beta_1+\beta_2)^{-s + 2}}{\beta_1\beta_2(s-1)(s-2)} - \frac{(\beta_1+\theta\beta_2)^{-s+2}}{(\beta_1+\theta\beta_2)\beta_2(s-1)(s-2)}\\
            & \qquad + \widetilde{R}(s,\beta_1,\beta_2,\theta)
    \end{align*}
    where $\widetilde{R}$ is holomorphic in $s$ for $\Re(s) > 1$. Hence, the residue at $s = 2$ is $\frac{\theta}{\beta_1(\beta_1 + \theta \beta_2)}$. 
Taking the derivatives with respect to $\beta_1$ and $\beta_2$, we obtain,
\begin{align*}
    - \frac{\partial}{\partial \beta_1}
    \sum_{\substack{x_1,x_2 \geq 1\\ x_2 \leq \theta x_1}} (\beta_1 x_1 + \beta_2 x_2)^{-s}
    %& = \frac{\theta(2\beta_1 + \theta \beta_2)}{\beta_1^2(\beta_1+\theta \beta_2)^2} \frac{1}{s-2}+ \dots \\
    & = \frac{1}{\beta_1^2\beta_2}\frac{\theta(\theta + 2 \frac{\beta_1}{\beta_2})}{(\theta + \frac{\beta_1}{\beta_2})^2} \frac{1}{s-2} + R_1(s,\beta_1,\beta_2,\theta)
\end{align*}
and similarly
\begin{align*}
    - \frac{\partial}{\partial \beta_2}
    \sum_{\substack{x_1,x_2 \geq 1\\ x_2 \leq \theta x_1}} (\beta_1 x_1 + \beta_2 x_2)^{-s}
    %& = \frac{\theta^2}{\beta_1(\beta_1+\theta \beta_2)^2} \frac{1}{s-2} + \dots \\
    & = \frac{1}{\beta_1 \beta_2^2} \frac{\theta^2}{(\theta + \frac{\beta_1}{\beta_2})^2} \frac{1}{s-2} + R_2(s,\beta_1,\beta_2,\theta)
\end{align*}
where both remainder terms $R_1$ and $R_2$ are holomorphic in $s$ in the half-plane $\sigma := \Re(s) > 1$ and are bounded, up to positive constants, by
\[
    \frac{|s|^2}{\sigma - 1} \min(\beta_1,\beta_2)^{-\sigma-1}.
\]
This decrease makes it possible to apply the residue theorem in order to shift to the left the vertical line of integration from $\sigma = 3$ to $\sigma = \frac{3}{2}$.
When $\beta_1$ and $\beta_2$ tend to $0$ and $\frac{\beta_1}{\beta_2}$ tends to $\ell$, we thus find
\begin{align*}
    \EE_{\beta,\lambda}[X_1^\theta] &= \frac{\zeta(3)-\Li_3(1-\lambda)}{\zeta(2)}\left[ \frac{1}{\beta_1^2\beta_2} \frac{\theta(\theta + 2\frac{\beta_1}{\beta_2})}{(\theta + \frac{\beta_1}{\beta_2})^2} + O\left(\frac{1}{|\beta|^{5/2}}\right)\right],\\
    \EE_{\beta,\lambda}[X_2^\theta] &= \frac{\zeta(3)-\Li_3(1-\lambda)}{\zeta(2)} \left[\frac{1}{\beta_1^2\beta_2} \frac{\theta^2}{(\theta + \frac{\beta_1}{\beta_2})^2} + O\left(\frac{1}{|\beta|^{5/2}}\right)\right].
\end{align*}
We obtain the announced result by normalizing these quantities by their limits when $\theta$ goes to infinity.
\end{proof}

\begin{lemma}[Uniform exponential concentration]
    \label{lem:concentration}
 Suppose that $\beta_1$ and $\beta_2$ tend to $0$ such that $\beta_1 \asymp \beta_2$ and $\lambda$ is bounded from above. For all $\eta \in (0,1)$, we have
    \[
        \PP_{\beta,\lambda}\left[\sup_{1\leq i \leq 2}\sup_{\theta \in [0,\infty]} \frac{|X^\theta_i - m_i^\theta(\beta,\lambda)|}{m_i^\infty(\beta,\lambda)} > \eta\right] \leq
        \exp\left\{-\frac{c(\lambda)\eta^2}{8\beta_1\beta_2}\left(1 + o(1)\right)\right\}.
    \]
\end{lemma}

\begin{proof}
    Fix $i \in \{1,2\}$ and let $M_\theta = X_i^\theta - m_i^\theta(\beta,\lambda)$ for all $\theta \geq 0$. The stochastic process $(M_\theta)_{\theta \geq 0}$ is a $\PPbl$-martingale, therefore $(e^{t M_\theta})_{\theta \geq 0}$ is a positive $\PPbl$-submartingale for any choice of $t \geq 0$ such that $\EEbl[e^{tX_i}]$ is finite. This condition is satisfied when $t < \beta_1$. Doob's martingale inequality implies for all $\eta > 0$,
    \begin{align*}
        \PPbl\left[\sup_{\theta \in [0,\infty]} M_\theta > \eta\,m_i^\infty(\beta,\lambda)\right]
        & = \PPbl\left[\sup_{\theta \in [0,\infty]} e^{tM_\theta}
            > e^{t\eta m_i^\infty(\beta,\lambda)}\right]\\
        & \leq e^{-t\eta m_i^\infty(\beta,\lambda)} \,\EEbl\left[e^{tM_\infty}\right] = e^{-t(\eta+1)m_i^\infty(\beta,\lambda)} \,\EEbl[e^{tX_i}]
    \end{align*}
    For $i = 1$, Lemma~\ref{lem:derivatives_Z} shows that the logarithm of the right-hand side satisfies
    \[
        -t(1+\eta) m_1^\infty(\beta,\lambda) + \log \frac{Z(\beta_1 - t, \beta_2 ; \lambda)}{Z (\beta_1,\beta_2;\lambda)} = \frac{c(\lambda)}{\beta_1\beta_2}\left[-\frac{t(1+\eta)}{\beta_1} - 1 + \frac{\beta_1}{\beta_1-t} + o(1)\right]
    \]
    asymptotically when $t$ and $\beta_1$ are of the same order. The same holds for $i=2$. This is roughly optimized for the choice $t = \beta_i\left(1-(1+\eta)^{-1/2}\right)$, which gives
    \[
        \PP_{\beta,\lambda} \left[\sup_{\theta \in[0,\infty]} M_\theta > \eta\, m_i^\infty(\beta,\lambda)\right] \leq \exp\left\{-\frac{2c(\lambda)}{\beta_1\beta_2}\left(1+\frac{\eta}{2}-\sqrt{1+\eta} + o(1)\right)\right\}.
    \]
    When considering the martingale defined by $N_\theta = m_i^\theta(\beta,\lambda) - X_i^\theta$, one obtains with the same method
    \[
        \PPbl\left[\sup_{\theta \in [0,\infty]} N_\theta > \eta\,m_i^\infty(\beta,\lambda)\right] \leq \exp\left\{-\frac{2c(\lambda)}{\beta_1\beta_2}\left(1 - \frac{\eta}{2} - \sqrt{1-\eta} + o(1)\right)\right\}.
    \]
    Since the previous inequalities hold for both $i\in\{1,2\}$, a simple union bound now yields
    \[
        \PPbl\left[\sup_{1\leq i \leq 2}\sup_{\theta \in [0,\infty]} \frac{|X^\theta_i - m_i^\theta(\beta,\lambda)|}{m_i^\infty(\beta,\lambda)} > \eta\right] \leq 4 \exp\left\{-\frac{c(\lambda)\eta^2}{8\beta_1\beta_2}\left(1+ o(1)\right)\right\}.
    \]
\end{proof}

We introduce the following parametrization of the arc of parabola $\sqrt{y} + \sqrt{1-x} = 1$:
\[
    x_1(\theta) = \frac{\theta(\theta+2)}{(\theta + 1)^2}, \quad x_2(\theta) = \frac{\theta^2}{(\theta + 1)^2}, \qquad \theta \in [0,\infty].
\]

\begin{theorem}[Limit shape for numerous vertices]
	\label{thm:limit_shape_numerous}
    Assume that $n_1 \asymp n_2 \to +\infty$, and $k = O(|n|^{2/3})$, and $\log |n| = o(k)$.
   There exists $c > 0$ such that for all $\eta \in (0,1)$, 
    \[
        \PP_{n,k}\left[\sup_{1\leq i \leq 2}\sup_{\theta \in [0,\infty]} \frac{|X^\theta_i - x_i(\frac{\beta_2}{\beta_1}\theta)|}{n_i} > \eta\right] \leq
        \exp\left\{-c\eta^2 k\left(1+ o(1)\right)\right\}.
    \]
    In particular,  the Hausdorff distance between a random convex polygonal line on $\frac{1}{n}\ZZ_+^2$ joining $(0,0)$ to $(1,1)$ with at most $k$ vertices and the arc of parabola $\sqrt{y}+ \sqrt{1-x} = 1$ converges in probability to $0$.
\end{theorem}

\begin{proof}
    Using the triangle inequality and Lemma~\ref{lem:expectation}, we need only prove the analogue of Lemma~\ref{lem:concentration} for the uniform probability $\PP_{n,k}$. Remind that the measure $\PPbl$ conditional on the event $\{X=n,K=k\}$ is nothing but the uniform probability $\PP_{n,k}$. Hence for all event $E$,
    \[
        \PP_{n,k}(E) \leq \frac{\PPbl(E)}{\PPbl(X=n,K=k)}.
    \]
    Applying this with the deviation event above for the parameters $(\beta,\lambda)$ defined in section~\ref{sec_conv_calibration} and using the Local Limit Theorem~\ref{thm:local_limit} as well as the concentration bound provided by Lemma~\ref{lem:concentration}, the right-hand side reads, up to constants,
    \[
        \frac{|n|^2}{\sqrt{k}} \,\exp\left\{-c\eta^2 k (1+o(1))\right\}.
    \]
    Since $\log |n| = o(k)$, the result follows.
\end{proof}

\section{Convex lattice polygonal lines with few vertices}
\label{sec:few}

\subsection{Combinatorial analysis}
The previous machinery does not apply in the case of very few vertices but it can be completed by an an elementary approach that we present now which will actually work up to a number of vertices negligible compared to $n^{1/3}$. It is based on the following heuristics: when $n$ tends to $+\infty$ and the number of edges $k$ is very small compared to $n$, one can expect that choosing an element of $\Pi(n;k)$ at random is somewhat similar to choosing $k-1$ vertices from $[0,1]^2$ in convex position at random. 
\Barany~\cite{barany_sylvesters_1999} and \Barany, Rote, Steiger, Zhang~\cite{barany_central_2000} proved by two different methods the existence of a parabolic limit shape in this continuous setting. These works are based on Valtr's observation that each convex polygonal line with $k$ edges is associated, by permutation of the edges, to exactly $k!$ increasing North-East polygonal lines with pairwise different slopes.
%\[
%    \begin{array}{ccc}
%        \Pi(n;k) & \longrightarrow & \mathcal{P}_{k-1}(\{1,\dots,n\}) \times \mathcal{P}_{k-1}(\{0,\dots,n-1\})\\
%        ((u_i,v_i))_{i=0}^k & \longmapsto & (\{ u_i, \; 1 \leq i \leq k-1\} , \{ v_i, \;\, 1 \leq i \leq k-1\})
%    \end{array}
%\]
%
%
%

Our first theorem is the convex polygonal line analogue to a result of Erd{\"o}s and Lehner on integer partitions \cite[Theorem~4.1]{erdos_distribution_1941}.
\begin{theorem}
\label{thm:asymp_very_few_vertices}
    The number of convex polygonal lines joining $(0,0)$ to $(n,n)$ with $k$ edges satisfies
    \[
     p(n;k) = \frac{1}{k!}\binom{n-1}{k-1}^2\left(1+o(1)\right),
    \]
    this formula being valid uniformly in $k$ for $k = o(n^{1/2}/(\log n)^{1/4})$.
\end{theorem}

\begin{proof}
    Let us start by proving an upper bound. This is done by considering the inequality
    \[
        \left| \Pi(n;k) \right| \leq \frac{1}{k!}\binom{n-1}{k-1}^2 + \frac{2}{(k-1)!}\binom{n-1}{k-2}\binom{n-1}{k-1} + \frac{1}{(k-2)!} \binom{n-1}{k-2}^2
    \]
    where the first term bounds the number of convex polygonal lines which are associated to strictly North-East lines, the second term bounds the number of lines having either a first horizontal vector or a last vertical one, and the third term bounds the numbers of convex polygonal lines having both a horizontal and a vertical vector. 

We now turn to a lower bound. Let $\{U_1, U_2, \dots, U_{k-1}\}$ and $\{V_1, V_2, \dots, V_{k-1}\}$ be two independent and uniformly distributed random subsets of $\{1,\dots,n-1\}$ of size $k-1$ whose elements are indexed in increasing order $U_1 < U_2 < \cdots < U_{k-1}$ and $V_1 < V_2 < \cdots <  V_{k-1}$. Let $M_0 = (0,0)$, $M_k = (n,n)$ and $M_i = (U_i,V_i)$ for $1 \leq i \leq k-1$. Obviously, the polygonal line $(M_0,M_1,\dots,M_n)$ has uniform distribution among all increasing polygonal line from $(0,0)$ to $(n,n)$.
We claim that the distribution of
$(\overrightarrow{M_{0}M_1},\overrightarrow{M_1M_2},\dots,\overrightarrow{M_{k-1}M_k})$ conditioned on the event that no two of these vectors are parallel is uniform among the lines of $\Pi(n,k)$ such that no side is parallel to the $x$-axis or the $y$-axis. Moreover, since the vectors are exchangeable, the probability that we can find $i < j$ such that $\overrightarrow{M_{i-1}M_i}$ and $\overrightarrow{M_{j-1}M_j}$ are parallel is bounded from above by $\binom{k}{2}$ times the probability that
$Y = \overrightarrow{M_0M_1}$ and $Z = \overrightarrow{M_1M_2}$ are parallel. Using the simple estimate
\[
    \binom{n-1}{k-1} \geq \frac{n^{k-1}}{(k-1)!}(1-o(1))
\]
which is asymptotically true since $k = o(\sqrt{n})$, we find that for all $(y,z) \in (\N^2)^2$, the probability that $Y=y$ and $Z=y$ is
\begin{align*}
    \PP(Y = y, Z=z) &= \frac{\binom{n-y_1-z_1}{k-3}\binom{n-y_2-z_2}{k-3}}{\binom{n-1}{k-1}^2}\\
    & \leq \frac{4k^2}{n^2}\left(1 - \frac{y_1+z_1}{n}\right)_+^{k-3}\left(1-\frac{y_2+z_2}{n}\right)_+^{k-3}\\
    & \leq \frac{4k^2}{n^2} \exp\left\{-\frac{k-3}{n}\left(y_1+y_2+z_1+z_2\right)\right\}.
\end{align*}
We can therefore dominate the probability that $Y$ and $Z$ are parallel by the probability that geometrically distributed random vectors are parallel, which is exactly estimated in the following lemma applied with $\beta = \frac{k}{n}$. In conclusion, the probability that at least two vectors are parallel is bounded by $\frac{k^4}{n^2}\log(n)$ up to a constant.
\end{proof}

\begin{lemma}
    Let $Y_1,Y_2,Z_1,Z_2$ be independent and identically distributed geometric random variables of parameter $1-e^{\beta}$ with $\beta > 0$. When $\beta$ goes to $0$, the probability that the vectors $Y=(Y_1,Y_2)$ and $Z = (Z_1,Z_2)$ are parallel is asymptotically equal to
    \[
        \frac{\beta^2}{\zeta(2)}\log \frac{1}{\beta}.
    \]
\end{lemma}

\begin{proof}
    The probability that $Y$ and $Z$ are parallel is
    \[
        \sum_{x\in \XX}\sum_{i,j \geq 1} \PP(Y = i\,x, Z = j \,x) = (1-e^{-\beta})^4 \sum_{x\in\XX}\sum_{i,j \geq 1} e^{-\beta(i+j)(x_1 + x_2)}.
    \]
    The Mellin transform of the double summation in the right-hand side with respect to $\beta > 0$ is well-defined for all $s \in \CC$ with $\Re(s) > 2$ and it is equal to
    \[
        \sum_{x\in \XX}\sum_{i,j \geq 1} \frac{\Gamma(s)}{(x_1+x_2)^s(i+j)^s} = \frac{\Gamma(s)}{\zeta(s)}(\zeta(s-1)-\zeta(s))^2.
    \]
    Expanding this Mellin transform in Laurent series at the pole $s = 2$ of order $2$ and using the residue theorem to express the Mellin inverse, one finds
    \[
        \sum_{x\in \XX}\sum_{i,j \geq 1} e^{-\beta(i+j)(x_1+x_2)} = \frac{1}{\zeta(2)}\frac{\log \frac{1}{\beta}}{\beta^2} - \frac{C}{\beta^2} + O\left(\frac{1}{\beta}\right),\qquad \text{as }\beta\to 0.
    \]
    where $C = \frac{2\zeta(2) - \zeta'(2) - 1 - \gamma}{\zeta(2)} \approx 0.471207$.
\end{proof}

\subsection{Limit shape}

\begin{theorem}[Limit shape for few vertices]
\label{thm:limit_shape_few}
The Hausdorff distance between a random convex polygonal line in $(\frac{1}{n}\ZZ \cap [0,1])^2$ joining $(0,0)$ to $(1,1)$ having at most $k$ vertices and the arc of parabola $\sqrt{\vphantom{x}y}+ \sqrt{\vphantom{y}1-x} = 1$ converges in probability to $0$ when both $n$ and $k$ tend to $+\infty$ with $k = o(n^{1/3})$.
\end{theorem}

\begin{proof}
\Barany~\cite{barany_sylvesters_1999} and \Barany, Rote, Steiger, Zhang~\cite{barany_central_2000} proved by two different methods the existence of a limit shape in the following continuous setting: if one picks at random $k-1$ points uniformly from the square $[0,1]^2$, then conditional on the event that these points are in convex position, the Hausdorff distance between the convex polygonal line thus defined and the parabolic arc goes to $0$ in probability as $k$ goes to $+\infty$. Our strategy is to show that this result can be extended to the discrete
setting $([0,1] \cap \frac{1}{n}\ZZ)^2$ if $k$ is small enough compared to $n$ by using a natural embedding of the discrete model into the continuous model.

For this purpose, we first observe that the distribution of the above continuous model can be described as follows: pick uniformly at random $k-1$ points from both the $x$-axis and the $y$-axis, rank them in increasing order and let $0 = U_0  < U_1 <
U_2 < \dots < U_{k-1} < U_k = 1$ and $0 = V_0 < V_1 < V_2 < \dots < V_{k-1} < V_k = 1$ denote this ranking. The points $(U_i,V_i)$ define an increasing North-East polygonal line joining $(0,0)$ to $(1,1)$. Reordering the segment lines of this line by increasing slope order, exchangeability arguments show that we obtain a convex line with $k$ edges that follows the desired distribution. This is analogous to the discrete construction of strictly North-East convex lines from $(0,0)$ to $(n,n)$ that occurs in the proof of Theorem~\ref{thm:asymp_very_few_vertices}.

Now, we define the lattice-valued random variables $\tilde{U}_0 \leq \tilde{U}_1 \leq \tilde{U}_2 \leq \dots \leq \tilde{U}_{k-1} \leq \tilde{U}_{k}$ and  $\tilde{V}_0  \leq \tilde{V}_1 \leq \tilde{V}_2 \leq \dots \leq \tilde{V}_{k-1} \leq \tilde{V}_k$ by discrete approximation:
\[
    \begin{cases}
        \tilde{U}_i \in \frac{1}{n}\ZZ, \quad U_i \leq \tilde{U}_i < U_i + \frac{1}{n}\\
        \tilde{V}_i \in \frac{1}{n}\ZZ, \quad V_i - \frac{1}{n} < \tilde{V}_i \leq V_i,
    \end{cases}\qquad \text{for }1 \leq i \leq k-1.
\]
Remark that we still have $(\tilde{U}_0,\tilde{V}_0) = (0,0)$ and $(\tilde{U}_k,\tilde{V}_k) = (1,1)$.

Let $X_i = (U_i - U_{i-1} ,V_i - V_{i-1})$ and let $\tilde{X}_i = (\tilde{U}_i - \tilde{U}_{i-1} ,\tilde{V}_i - \tilde{V}_{i-1})$ be the discrete approximation of $X_i$ for $1 \leq i \leq k$. Conditional on the event that the slopes of $(X_1,\dots,X_k)$ and $(\tilde{X}_1,\dots, \tilde{X}_k)$ are pairwise distinct and ranked in the same order, the Hausdorff distance between the associated convex polygonal lines is bounded by $\frac{k}{n}$, which goes asymptotically to $0$. Since a direct
application of \cite[Theorem~2]{barany_central_2000} shows that the distance between the convex line associated to $X$ and the parabolic arc converges to $0$ in probability as $k$ tends to $+\infty$, we deduce
that the Hausdorff distance between the convex line associated to $\tilde{X}$ and the parabolic arc also converges in probability to $0$ on this event.
As in the proof of Theorem~\ref{thm:asymp_very_few_vertices}, the joint density of $(X_i,X_j)$ is dominated by the density of a couple of independent vectors
whose coordinates are independent exponential variables with parameter $k$. These vectors being of order of magnitude $\frac{1}{k}$, the order of the slopes of $(X_i,X_j)$ and $(\tilde{X}_i,\tilde{X}_j)$ may be reversed only if the angle between $X_i$ and $X_j$ is smaller than $\frac{ck}{n}$ for some $c > 0$, which happens with probability of order $\frac{k}{n}$. Consequently, the probability that there exists $i < j$ for which the slopes of $(X_i,X_j)$ and $(\tilde{X}_i,\tilde{X}_j)$ are ranked in opposite is bounded, up to a constant, by $\binom{k}{2} \frac{k}{n}$. Therefore, the Hausdorff distance between the convex line associated to $\tilde{X}$ and the parabolic arc also converges to $0$ in probability if $k = o(n^{1/3})$.

The final step is to compare the distribution of the increasing reordering of $(\tilde{X}_1,\dots,\tilde{X}_k)$ with the uniform distribution on $\Pi(n;k)$. As a consequence of Theorem~\ref{thm:asymp_very_few_vertices}, the probability that a uniformly random element of $\Pi(n;k)$ is strictly North-East tends to $1$. The key point, which follows from Valtr's observation, is that the uniform distribution on strictly North-East convex lines with $k$ edges coincides with the distribution of the line obtained by reordering the vectors $(\tilde{X}_1,\dots,\tilde{X}_k)$, conditional on the event that these vectors are pairwise linearly independent and strictly North-East. Since we showed in the previous paragraph that all the angles between two vectors of $(\tilde{X}_1,\dots,\tilde{X}_k)$ are at least $\frac{ck}{n}$ with probability $1 - O(\frac{k^3}{n})$, the linear independence condition occurs with probability tending to $1$.  On the other hand, $(\tilde{X}_1,\dots,\tilde{X}_k)$ are strictly North-East with probability $1 - O(\frac{k^2}{n})$. Therefore, the event we conditioned on has a probability tending to $1$, which proves that the total variation distance between the two distributions tends to $0$.
\end{proof}

\section{Back to \Jarnik's problem}
\label{sec:jarnik}

In \cite{jarnik_uber_1926}, \Jarnik gives an asymptotic formula of the maximum possible number of vertices of a convex lattice polygonal line having a \emph{total Euclidean length} smaller than $n$, and whose segments make an angle with the $x$-axis between $0$ and $\frac{\pi}{4}$. What he finds is $\frac{3}{2}\,\frac{n^{2/3}}{(2\pi)^{1/3}}$. If, in order to be closer to our setting, we  ask the segments to make an angle with the $x$-axis between $0$ and $\frac{\pi}{2}$, \Jarnik's formula is
changed into $\frac{3}{2}\frac{n^{2/3}}{\pi^{1/3}}$ (which is  twice the above result for $\frac{n}{2}$). 

In this section, we want to present a detailed combinatorial analysis of this set of lines, which leads to  \Jarnik's result as well as to the asymptotic of the \emph{typical} number of vertices of such lines. It is the analog
of \Barany, \Sinai and Vershik's result when the constraint concerns the total length.

Let us first describe \Jarnik's argument, which is a good application of the correspondence described in section~\ref{sec:correspondence}. It says the following: the function $\omega$ realizing the maximum can be taken among the functions taking their values in $\{0,1\}$. Indeed, by changing the non-zero values of a function $\nu$ into $1$, one can obtain a polygonal line with the same number of vertices, but with a shorter length. Now, if the number of vertices $k$ is given, the convex line having minimal
length, will be defined by the function $\omega$ which associates $1$ to the $k$ points  of $\XX$ which are the closest to the origin.
Since the set $X$ has an asymptotic density $\frac{6}{\pi^2}$, when  $N$ is big, this set  of points is asymptotically equivalent to the intersection of $X$ with the disc of center $O$ having radius $R$ satisfying $\frac{6}{\pi^2}\cdot  \frac{\pi R^2}{4}=N$ i.e. $R=(\frac{2\pi}{3} N)^{1/2}$. The total length of the line is equivalent to 
$L= \int_0^R r\times \frac{6}{\pi^2}\frac{\pi}{2}r dr=\frac{R^3}{\pi}=\frac{(\frac{2\pi}{3} N)^{3/2}}{\pi}$. 
This yields precisely $N=\frac{3}{2} \frac{L^{2/3}}{\pi^{1/3}}\simeq1.02\, L^{2/3}$. 

In order to get finer results, we introduce the probability distribution on the space $\Omega$ proportional to
\[
    \exp\left(-\beta \sum_{x \in \XX} \omega(x) \sqrt{|x_1|^2 + |x_2|^2}\right) \lambda^{\sum_{x\in \XX} 1_{\{\omega(x) > 0\}}}
\]
which depends on two parameters $\beta,\lambda$. In this set-up, the partition function turns out to be
\[
    Z = \prod_{x\in \XX} \frac{1-(1-\lambda)e^{-\beta \sqrt{|x_1|^2 + |x_2|^2}}}{1-e^{-\beta \sqrt{|x_1|^2+|x_2|^2}}}.
\]
The Mellin transform representation for $\log Z$ now involves
\[
    \frac{\Gamma(s)(\Li_{s+1}(1-\lambda)-\zeta(s+1))}{\zeta(s)}\sum_{x_1,x_2 \geq 1} {(|x_1|^2+|x_2|^2)}^{-s/2}, \qquad \Re(s) > 2.
\]
The factors $\zeta(s)^{-1}$ and $\Li_{s+1}(1-\lambda)-\zeta(s+1)$, which correspond respectively to the coprimality condition on the lattice and to the penalty of vertices, are still present. The main difference relies in the replacement of the Barnes zeta function by the Epstein zeta function which comes from the penalty by length in the model. With the help of the residue analysis of this Mellin transform and a local limit theorem, we obtain:

\begin{theorem}
    Let $p_J(n;k)$ denote the number of convex polygonal lines on $\ZZ_+^2$ issuing from $(0,0)$ with $k$ vertices and length between $n$ and $n+1$. As $n$ tends to $+\infty$,
    \[
        \text{if}\qquad \frac{k}{n^{2/3}} \longrightarrow \frac{\pi^{1/3}}{2}\cc(\lambda), \qquad\text{then}\qquad  \frac{1}{n^{2/3}} \log p_J(n;k) \longrightarrow \frac{\pi^{1/3}}{2}\e(\lambda),
    \]
    where $\e$ and $\cc$ are the functions introduced in Theorem~\ref{thm:detailed_comb}. Moreover, the Hausdorff distance between a random element of this set normalized by $\frac{1}{n}$, and the arc of circle $\{(x,y) \in [0,1]^2 \mid x^2 + (y-1)^2 = 1\}$ converges to $0$ in probability.
\end{theorem}

From this result, we deduce that the typical number of vertices of such a line which is achieved for $\lambda=1$ is asymptotically equal to
\[
    \left(\frac{3}{4\pi\zeta(3)^2}\right)^{1/3} n^{2/3}.
\]
Similarly, the total number of convex lattice polygonal lines having length between $n$ and $n+1$ is asymptotically equal to
\[
    \exp\left(\frac{3^{4/3}\zeta(3)^{1/3}}{(4\pi)^{1/3}}\, n^{2/3} (1+o(1))\right).
\]
In addition, we can derive \Jarnik's result in the lines of Remark~\ref{rem:longest}.

\section{Mixing constraints and finding new limit shapes}
\label{sec:onion}

In this section we introduce a family of convex lattice polygonal line models which achieves a continuous interpolation of limit shapes between the diagonal of the square and the South-East corner sides of the square, passing through the arc of circle and the arc of parabola. Let $\|\cdot\|_1$ and $\|\cdot\|_2$ denote respectively the Taxicab norm and the Euclidean norm on $\RR^2$. Recall that for all $x \in \RR^2$,
\[
    \|x\|_1 = |x_1| + |x_2| \geq \|x\|_2 = \sqrt{|x_1|^2+|x_2|^2} \geq \frac{1}{\sqrt{2}}\|x\|_1.
\]

The Gibbs distribution we consider on the space $\Omega$ involves both these norms in order to take into account both the extreme point of the line and its length:
\[
    \frac{1}{Z}\exp\left(-\beta \sum_{x\in\XX} \omega(x) (\|x\|_1 + \lambda\sqrt{2} \|x\|_2)\right), \quad
    Z = \prod_{x \in \XX} \left(1 - e^{-\beta(\|x\|_1 + \lambda\sqrt{2}\|x\|_2)}\right).
\]
This infinite product is convergent if $\beta >0$ and $\lambda > -\frac{1}{\sqrt{2}}$ or if $\beta < 0$ and $\lambda < -1$.
In both cases, the Mellin transform representation of $\log Z$ involves
\[
    \frac{\Gamma(s)\zeta(s+1)}{\zeta(s)} \sum_{x_1,x_2 \geq 1} (\|x\|_1 + \lambda\sqrt{2} \|x\|_2)^{-s},\qquad \Re(s) > 2.
\]
As usual, the leading term of the expansion of $\log Z$ when $\beta \to 0$ is obtained by computing the residue of this function at $s = 2$. It turns out to be
\[
    \frac{\zeta(3)}{2\zeta(2)} \int_{-\pi/4}^{\pi/4} \frac{d\theta}{(\lambda + \cos (\theta))^2}.
\]
An application of the residue theorem shows that the expected length of the curve is asymptotically equivalent to
\[
    \frac{1}{\beta^3}\frac{\zeta(3)}{\sqrt{2}\zeta(2)} \int_{-\pi/4}^{\pi/4} \frac{d\theta}{(\lambda + \cos (\theta))^3}
\]
and that the coordinates of the ending point have asymptotic expected value
\[
    \frac{1}{\beta^3}\frac{\zeta(3)}{2\zeta(2)} \int_{-\pi/4}^{\pi/4} \frac{\cos(\theta) d\theta}{(\lambda + \cos (\theta))^3}.
\]

As in previous sections, a local limit theorem gives a correspondence between this Gibbs measure and the uniform distribution on a specific set of convex lines, namely the convex polygonal line with endpoint $(n,n)$ and total length belonging to $[L\cdot n, L\cdot n + 1]$ for some $L \in ]\sqrt2, 2[$ which is a function of $\lambda$,
\[
    L(\lambda) = \sqrt2 \dfrac{\int_0^{\frac \pi 4}{\frac{1}{(\lambda+\cos u)^3}}du}{\int_0^{\frac \pi 4}\frac{\cos u}{(\lambda+\cos u)^3}du}.
\]
By computations analogous to section~\ref{sec:limit_shape}, one can show that the uniform distribution on lines with length between $L(\lambda)\cdot n$ and
$L(\lambda) \cdot n+1$ concentrates around the curve described by the parametrization
$$x_\lambda(\phi)=\sqrt2\dfrac{\int_{0}^\phi{\frac{\cos u}{(\lambda+\cos (u-{\frac \pi 4}))^3}}du}{\int_{-\pi/4}^{\pi/4}{\frac{\cos u}{(\lambda+\cos u)^3}du}},\quad y_\lambda(\phi)=\sqrt2\dfrac{\int_{0}^\phi{\frac{\sin u}{(\lambda+\cos (u-{\frac \pi 4}))^3}}du}{\int_{-\pi/4}^{\pi/4}{\frac{\cos u}{(\lambda+\cos u)^3}}du} \quad (0\leq\phi\leq{\frac \pi 2}).$$
The table provided in Figure~\ref{fig:table} summarizes the limit shapes that we obtain for some limit values of $\lambda$. See also Figure~\ref{fig:onion} for a plot showing the interpolation of those limit shapes.

\begin{figure}[h]
\begin{center}
    \renewcommand{\arraystretch}{2}
    \begin{tabular}{|c|cc||ccc|}
        \hline
        $\lambda$ & $-\infty$ & $-1$ & $-\dfrac{1}{\sqrt{2}}$ & $0$ & $+\infty$\\
        \hline
        Limit shape & circle & diagonal & square & parabola & circle\\
        Length $L(\lambda)$ & $\dfrac{\pi}{2}$ & $\sqrt{2}$ & $2$ & $1+\dfrac{\ln(1+\sqrt{2})}{\sqrt{2}}$ & $\dfrac{\pi}{2}$\\
        \hline
    \end{tabular}
\end{center}
\caption[Special values of the new limit shapes]{Critical and special values in the spectrum of limit shapes for the model of convex lattice lines with mixed constraints.}
\label{fig:table}
\end{figure}

\begin{figure}[h]
\begin{center}
    \includegraphics{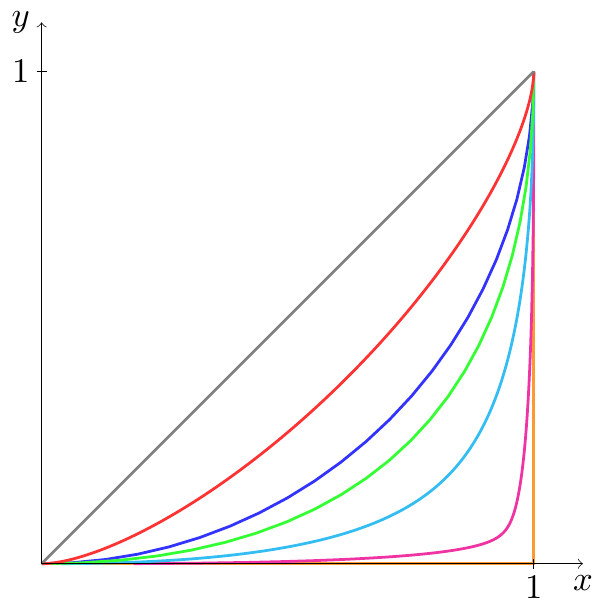}
    \caption[New limit shapes of different Euclidean lengths]{Limit shapes of different Euclidean lengths.
      Successively: \textcolor{black!50}{\mbox{$\sqrt{2}$ (diagonal)}}; \textcolor{red!80}{$1.48$}, \textcolor{blue!80}{$\frac{\pi}{2}$ (circle)}, \textcolor{green}{$1+\frac{\ln(1+\sqrt{2})}{\sqrt{2}}$ (parabola)}, \textcolor{cyan!80}{$1.72$}, \textcolor{magenta!80}{$1.89$} and \textcolor{orange!80}{$2$ (square)}.}
    \label{fig:onion}
\end{center}
\end{figure}

\pagebreak

\bibliography{convex-lines.bib}

\end{document}